\documentclass[a4paper,reqno,11pt]{article}
\usepackage{amsmath,amssymb,amsthm,booktabs,cases}
\usepackage{mathrsfs}
\usepackage{enumerate,framed,graphicx,array,color,multirow,mathrsfs}

\usepackage[utf8]{inputenc}
\usepackage{hyperref}
\hypersetup{
     colorlinks   = true,
     citecolor    = blue,
     linkcolor=blue,
     urlcolor=blue
}
\usepackage{bm}
\usepackage{euscript}
\usepackage{graphicx}
\usepackage{comment}
\usepackage{multicol}
\usepackage[usenames,dvipsnames,svgnames,table]{xcolor}

\usepackage{a4wide}

\usepackage{mathrsfs}
\usepackage{euscript}

\theoremstyle{plain}
\newtheorem{theorem}{Theorem}[section]

\newtheorem{proposition}[theorem]{Proposition}
\newtheorem{lemma}[theorem]{Lemma}

\theoremstyle{definition}
\newtheorem{remark}[theorem]{Remark}
\newtheorem{definition}[theorem]{Definition}

\numberwithin{equation}{section}

\newtheorem*{assumptions*}{\assumptionnumber}
\providecommand{\assumptionnumber}{}


\def\cB{\mathcal{B}}

\def\cF{\mathcal{F}}

\def\cN{\mathcal{N}}

\def\cT{\mathcal{T}}

\def\E{\mathbb E}

\def\P{\mathbb P}

\newcommand{\R}{\mathbb{R}}
\newcommand{\N}{\mathbb{N}}

\newcommand{\ud}{\ensuremath{ \mathrm{d}} }
\newcommand{\Ceil}[1]{\left\lceil #1 \right\rceil}

\newcommand{\FoxH}[5]{H_{#2}^{#1}\left(#3\:\middle\vert\: \begin{subarray}{l}#4\\[0.4em] #5\end{subarray}\right)}

\newcommand*{\one}{{{\rm 1\mkern-1.5mu}\!{\rm I}}}

\makeatletter
\@namedef{subjclassname@2020}{\textup{2020} Mathematics Subject Classification}
\makeatother

\title{Space-time fractional stochastic partial differential equations driven by L\'evy white noise}

\author{Yuhui Guo\footnote{Y. Guo: Guangdong Provincial/Zhuhai Key Laboratory of IRADS, and Department of Mathematical Sciences, Faculty of Science and Technology, Beijing Normal-Hong Kong Baptist University, Zhuhai, Guangdong, 519087, China. \textit{Email:} \url{yuhuiguo@uic.edu.cn}}
~~and~~Jiang-Lun Wu\footnote{J.-L. Wu: Guangdong Provincial/Zhuhai Key Laboratory of IRADS, and Department of Mathematical Sciences,
Faculty of Science and Technology, Beijing Normal-Hong Kong Baptist University, Zhuhai, Guangdong, 519087, China. \textit{Email:} \url{jianglunwu@uic.edu.cn}}
}

\date{}

\begin{document}

\maketitle

\begin{center}
\begin{minipage}[rct]{5 in}
\footnotesize \textbf{Abstract:}
This paper is concerned with the following space-time fractional stochastic nonlinear partial differential equation
\begin{equation*}
  \left(\partial_t^{\beta}+\frac{\nu}{2}\left(-\Delta\right)^{\alpha / 2}\right) u=I_{t}^{\gamma}\Big[ f(t,x,u)-\sum_{i=1}^{d} \frac{\partial}{\partial x_i} q_i(t,x,u)+ \sigma(t,x,u) F_{t,x}\Big]
\end{equation*}
for a random field $u(t,x):[0,\infty)\times\R^d \mapsto\R$, where $\alpha>0, \beta\in(0,2), \gamma\ge0, \nu>0, F_{t,x}$ is a L\'evy space-time white noise, $I_{t}^\gamma$ stands for the Riemann-Liouville integral in time, and $f,q_i,\sigma:[0,\infty)\times\R^d\times\R \mapsto\R$ are measurable functions. Under suitable polynomial growth conditions, we establish the existence and uniqueness of $L^2(\R^d)$-valued local solutions when the L\'evy white noise $F_{t,x}$ contains Gaussian noise component. Furthermore, for $p\in[1,2]$, we derive the existence and uniqueness of $L^p(\R^d)$-valued local solutions for the equation driven by pure jump L\'evy white noise. Finally, we obtain certain stronger conditions for the existence and uniqueness of global solutions.

\vspace{2ex}

\vspace{2ex}

\vspace{4ex}
\end{minipage}
\end{center}

\hypersetup{linkcolor=black}
\setcounter{tocdepth}{2}
\tableofcontents
\hypersetup{linkcolor=blue}

\section{Introduction}

In this article, we are concerned with the following nonlinear space-time fractional stochastic partial differential equations (FSPDEs) driven by a L\'evy type noise in $d$-dimensional space $\R^d$:
\begin{equation}\label{e:fde}
	\begin{cases}
	\displaystyle	\left(\partial_t^{\beta}+\frac{\nu}{2}\left(-\Delta\right)^{\alpha / 2}\right) u(t, x,\omega) =I_{t}^{\gamma}\Big[ f(t,x,u(t, x,\omega))&\\
    \displaystyle \qquad\quad -\sum_{j=1}^{d} \frac{\partial}{\partial x_j} q_j(t,x,u(t, x,\omega))+ \sigma(t,x,u(t, x,\omega)) F_{t,x}(\omega)\Big] ,& (t,x,\omega)\in[0,\infty)\times\mathbb{R}^d\times\Omega,   \\
		u(0, x,\omega)=u_0(x,\omega),                                                                                       & \text { if } \beta \in(0,1], \\
		u(0, x,\omega)=u_0(x,\omega), \quad \dfrac{\partial}{\partial t} u(0,x,\omega)=u_1(x,\omega),                                 & \text { if } \beta \in(1,2),
	\end{cases}
\end{equation}
with
\begin{equation*}
  \alpha>0,\quad \beta\in(0,2),\quad \gamma\ge 0, \quad \nu>0,
\end{equation*}
where initial (random) data $u_0,u_1$ are $\cF_0$-measurable, $F_{t,x}$ denotes a L\'evy  type noise, see Section \ref{se:levy} for more details of $F_{t,x}$.
Here for each $1\le j\le d$, $q_j:[0,\infty)\times\R^d\times\R\mapsto\R$ is measurable and corresponds to the ``nonlinearity", and $f,\sigma:[0,\infty)\times\R^d\times\R\mapsto\R$ are measurable.

In the above equation \eqref{e:fde}, $\left(-\Delta\right)^{\alpha / 2}$ is \textit{fractional} ($0<\alpha< 2$) or \textit{power} of  ($\alpha>2$) \textit{Laplacian} in space,
$\partial_t^\beta$ stands for the {\em Caputo fractional differential operator} in time:
\begin{equation}\label{e:CFDO}
  \partial_t^{\beta} f(t):=
\begin{cases}
  \displaystyle \frac{1}{\Gamma(n-\beta)}\int_{0}^{t}\frac{f^{(n)}(\tau)}{(t-\tau)^{\beta+1-n}}\ud\tau, & \mbox{if } \beta \neq n, \vspace{0.2cm} \\
  \dfrac{\ud^n}{\ud t^n}f(t),                                                                                    & \mbox{if }  \beta = n,
\end{cases}
\end{equation}
where $n:=\lceil\beta\rceil$ is the integer part of $\beta$ and
 $\Gamma(x) := \int_{0}^{\infty}e^{-t}t^{x-1}\ud t$ is the Gamma function,
$I_{t}^\gamma$  is the \emph{Riemann-Liouville integral} in time defined as follows
\begin{equation*}
  (I_t^{\gamma}f)(t):=\frac{1}{\Gamma(\gamma)} \int_{0}^{t}f(r)(t-r)^{\gamma-1}\ud r, \quad \text{ for }\gamma>0,
\end{equation*}
with  the convention $I_t^0=$Id (the identity operator), see \cite[Chapter 2]{Kilbas2006Theory} for more details.

As an important case of FSPDEs \eqref{e:fde}, the classical  Burgers equation:
\begin{equation*}
  \frac{\partial}{\partial t} u(t,x) = \frac{1}{2} \frac{\partial^2}{\partial x^2}u(t,x)-\frac{1}{2}\frac{\partial}{\partial x}(u(t,x))^2
\end{equation*}
originally formulated by Forsyth \cite{forsyth1906theory} in 1906 and Bateman \cite{bateman1915some} in 1915, serves as a fundamental model capturing the interplay between nonlinear convective phenomena and diffusive mechanisms. Burgers equations have also been utilised to investigate interface dynamics, as explored in the seminal work \cite{kardar1986dynamic}. A rigorous mathematical investigation of this equation was undertaken by Burgers \cite{Burgers1939Mathematical} in 1939. Fractional Burgers equations replace the Laplacian $\partial^2/\partial x^2$ by a (nonlocal)
operator $-(-\Delta)^{\alpha/2}$, to model anomalous transport and turbulent cascades in $\R^d$, see \cite{Biler1998Fractal,Woyczy1998Burgers} and  references therein.

Stochastic Burgers equations incorporate random forcing (Gaussian
 type or L\'evy type noises) to capture fluctuations beyond deterministic dynamics, see, for instance, \cite{Debbi2005Onthesolutions,Dong2007One-dimensional,Dong2023Global,Gy1999Onthe,Jacob2010Solving,Lewis2018Stochastic,Truman2006On}. Additive Gaussian space time white noise leads to mild solution frameworks amenable to the Hopf–Cole transform \cite{Bertini1994Thestochastic}, yielding explicit invariant measures, ergodic properties, and connections with the Kardar–Parisi–Zhang (KPZ) universality class \cite{Chan2000Scaling}.
 It is worth mentioning that \cite{Brze2007Onstochastic} undertakes a rigorous analysis of one-dimensional stochastic fractional nonlinear Burgers-type partial differential equations driven by Gaussian space time white noise, and establish the existence and uniqueness of their global solutions. \cite{Dong2007One-dimensional} proves the global existence and uniqueness of the strong, weak and mild solutions for one dimensional Burgers equation perturbed by a Poisson form process. \cite{Jacob2010Solving,Truman2006On} study non-linear stochastic equations of Burgers type driven by L\'evy space–time white noise in one space dimension, and then \cite{Wu2012On} extends the corresponding results for $d$-dimensional space when the noise is limited to pure jump case. \cite{Khanin2017Hyperbolicity} develops a completely different approach to extend invariant measures in the
one dimensional Burgers equation obtained by \cite{E2000Invariant} to arbitrary dimensional  case. Recently, \cite{Dong2023Global,Dong2024Ergodicity} studies stochastic Burgers equation on the three dimensional torus driven by  linear multiplicative noise,  and establish the global well-posedness.

The investigation of  FSPDEs \eqref{e:fde} involving the Caputo fractional derivative $\partial_t^{\beta}$  and the Riemann-Liouville integral $I_{t}^{\gamma}$, is primarily motivated by the mathematical and physical significance of incorporating memory effects and anomalous diffusion into stochastic systems, see some previous works \cite{Chen2017Nonlinear,Chen2024Moments,Chen2017Space-time,Chen2019Nonlinear,Guo2024Stochastic,Mijena2015Space} and references therein.
For the Cauchy problem with initial data prescribed at $t=0$, the Caputo fractional derivative $\partial_t^{\beta}$ is also a natural choice, see \cite[Chapters 1 and 3]{Diethelm2004Theanalysis} for  a detailed discussion. The incorporation of the fractional integral $I_{t}^{\gamma}$ induces a temporally non-local influence and yields a marked regularization effect, which can weaken the conditions of  existence and uniqueness of solutions as $\gamma$ increases, see \eqref{e:con-th1} and \eqref{e:con-mainth2}. Moreover, its introduction can also furnishes a more flexible modeling framework, thereby broadening the scope of potential applications.

In this article, we investigate the existence and uniqueness of both local and global solutions. We begin by establishing the local solutions under relatively weaker conditions; see Theorems \ref{th:1} and \ref{th:purejump}. By imposing stronger conditions, we can then prove the existence and uniqueness of the global solutions, see Theorem \ref{th:global}. Furthermore, when the equations reduce to the stochastic (fractional) Burgers equations, we compare our results in Theorems \ref{th:1}, \ref{th:purejump}, and \ref{th:global} with some known results in the literature, see Remarks \ref{re:comp-white} and \ref{re:comp-global}.

The rest of this paper is organised as follows. In Section \ref{se:pre}, we present some preliminaries on (pure jump) L\'evy space–time white noise, and then we recall and also derive several related results of the fundamental solution. In Section \ref{se:main-result}, we give the precise definitions of both local and global solutions,  and then establish the existence and uniqueness of solutions. Finally, Section \ref{se:proof} is devoted to presenting all relevant proofs.

{\bf Notation.} Throughout this paper, we use $\Vert\cdot\Vert_p$ to denote $L^p(\R^d)$-norm in the space variable. $\lceil\cdot\rceil$ is the ceiling function. For any Borel set $A$ in $d$-dimensional Euclidean space $\mathbb{R}^d$, we denote by $\lambda(A)$ its Lebesgue measuren for any $d\ge1$. We use the conventional notation $\N:=\{1,2,\dots\}$ and $i:=\sqrt{-1}$.  For any $\varphi\in L^1(\R^d)$, the Fourier transform of $\varphi$ is given by
\begin{equation*}
  \cF \varphi(\xi):= \int_{\R^d} \varphi(x) e^{-i\xi\cdot x}\ud x.
\end{equation*}
As usual, we will omit $\omega$ in the random field $u(t,x,\omega)$ and write  instead $u(t,x)$. We denote by $C$ a generic positive constant which might change in different positions.

\section{Preliminaries}\label{se:pre}
\subsection{L\'evy type noise}\label{se:levy}

We start with L\'evy space-time white noise on $[0,\infty)\times \R^d$.   Let $(\Omega,\cF,\P)$ be a given complete probability space and $(E,\cB(E),\mu)$ be an arbitrary $\sigma$-finite measure space. Following \cite[Theorem I.8.1]{Ikeda1981Stochastic}, there exists a Poisson random measure on the product measure  space
\begin{equation*}
  \left([0,\infty)\times \R^d\times E,\cB([0,\infty)\times \R^d)\times\cB(E),\ud t\ud x\otimes\mu \right),
\end{equation*}
i.e., a random measure
\begin{equation*}
  N:\cB([0,\infty)\times \R^d)\times\cB(E)\times\Omega \mapsto \N\cup\{0\}\cup\{\infty\},
\end{equation*}
 with mean measure
 $$\E[N([a,b]\times A,B,\cdot)]=|b-a|\:\lambda(A)\mu(B),$$
 for $[a,b]\times A\in\cB([0,\infty)\times\R^d)$ and $B\in \cB(E)$.  In fact, $N$ can be constructed canonically as follows:
 \begin{equation*}
   N([a,b]\times A,B,\omega):=\sum_{n\in\N} \sum_{j=1}^{\eta_n(\omega)}\one_{\{ ([a,b]\times A\cap K_n) \times (B\cap E_n)\}}(\xi_{n,j}(\omega))
   \one_{\{\omega\in\Omega: \eta_n(\omega)\ge1\}}(\omega)
 \end{equation*}
 for $[a,b]\times A\in\cB([0,\infty)\times \R^d)$, $B\in \cB(E)$ and $\omega\in\Omega$, where
 \begin{enumerate}[(a)]
   \item $\{K_n\}_{n\in\N}\subset \cB([0,\infty)\times \R^d)$ is a partition of $[0,\infty)\times \R^d$ with $0<\lambda(K_n)<\infty$, and $\{E_n\}_{n\in\N}\subset \cB(E)$ is a partition of $E$ with $0<\mu(E_n)<\infty$;
   \item for any $n,j\in\N$, $\xi_{n,j}:\Omega\to K_n\times E_n$ is $\cF / \mathcal{K}_n\times \mathcal{E}_n$-measurable with
         \begin{equation*}
           \P \left\{ \omega\in\Omega:\xi_{n,j}(\omega)\in [a,b]\times A\times B \right\}=\frac{|b-a|\lambda(A)\mu(B)}{\lambda(K_n)\mu(E_n)},
         \end{equation*}
         for $[a,b]\times A\in\mathcal{K}_n$ and  $B\in \mathcal{E}_n$, where $\mathcal{K}_n:=\cB([0,\infty)\times \R^d)\cap K_n$ and $\mathcal{E}_n:= \cB(E)\cap E_n$;
   \item for any $n\in\N$, $\eta_n:\Omega\to \N \cup\{0\}\cup\{\infty\}$ is a Poisson distributed random variable with
         \begin{equation*}
           \P \left\{ \omega\in\Omega: \eta_n(\omega)=k \right\}=\frac{e^{-|b-a|\lambda(A)\mu(B)}(|b-a|\lambda(A)\mu(B))^k}{k!};
         \end{equation*}
         for $k\in \N \cup\{0\}\cup\{\infty\}$;
   \item $\xi_{n,j}$ and $\eta_n$ are mutually independent for all $n,j\in\N$.
 \end{enumerate}

 \begin{remark}
 Give any $\sigma$-finite measure $\mu$ on $(E,\cB(E))$, there is always a Poisson random measure $N$ on the product measure space $([0,\infty)\times\R^d\times E, \cB([0,\infty)\times\R^d)\times \cB(E),\ud t\ud x \otimes \mu)$ which can be constructed in the above manner. Such a $N$ is called a canonical Poisson random measure associated with the product $\sigma$-finite measure $\ud t\ud x \otimes \mu$.
\end{remark}

Let $\{\cF_t\}_{t\ge0}$ be a right-continuous, increasing family of sub $\sigma$-algebras of $\cF$, each containing all $\P$-null sets of $\cF$,  such that the canonical Poisson random measure has the following properties:
\begin{enumerate}[(i)]
  \item $N([0,t]\times A,B,\cdot):\Omega\mapsto \N\cup\{0\}\cup\{\infty\}$ is $\cF_t / \mathcal{P}(\N \cup\{0\}\cup\{\infty\})$-measurable for all $(t,A,B)\in[0,\infty)\times\cB(\R^d)\times\cB(E)$, where $\mathcal{P}(\cdot)$ denotes the power set;
  \item $\left\{N([0,t+s]\times A,B,\cdot)-N([0,t]\times A,B,\cdot)\right\}_{s>0,(A,B)\in\cB(\R^d)\times\cB(E)}$  is independent of $\cF_t$ for any $t\ge 0$.
\end{enumerate}
For instance, we may directly take
\begin{equation*}
  \cF_t:=\sigma\left(\left\{ N([0,t]\times A,B,\cdot): (A,B)\in\cB(\R^d)\times\cB(E)\right\}\right)\vee \cN,\quad t\ge0,
\end{equation*}
where $\cN$ denotes all $\P$-null sets of $\cF$.

Define the compensating martingale measure
\begin{equation}\label{e:de-M}
  M(t,A,B,\omega):=N([0,t]\times A,B,\omega)-t\:\lambda(A)\mu(B),
\end{equation}
for any $(t,A,B)\in[0,\infty)\times\cB(\R^d)\times\cB(E)$ with $\lambda(A)\mu(B)<\infty$. Obviously,
\begin{equation*}
  \E\left[ M(t,A,B,\cdot) \right]=0,
\end{equation*}
and
\begin{equation*}
  \E\left[ \left|M(t,A,B,\cdot) \right|^2 \right]= t\:\lambda(A)\mu(B).
\end{equation*}

For any $\{\cF_t\}$-predictable integrand $h:[0,\infty)\times\R^d\times E\times\Omega\to\R$ which satisfies
\begin{equation*}
 \E\left[ \int_{0}^{t}\int_{A}\int_B \left|h(s,x,\xi,\cdot)\right| \ud s\ud x \mu(\ud\xi)\right]<\infty
\end{equation*}
for any $t\ge 0$ and some $(A,B)\in\cB(\R^d)\times\cB(E)$, we can define the stochastic integral
\begin{equation}\label{e:de-sto-int}
  \begin{split}
     \int_{0}^{t}\int_{A}\int_B h(s,x,\xi,\omega) M(\ud s,\ud x,\ud \xi,\omega):=&\int_{0}^{t}\int_{A}\int_B h(s,x,\xi,\omega) N(\ud s,\ud x,\ud \xi,\omega)\\
   &\quad- \int_{0}^{t}\int_{A}\int_B h(s,x,\xi,\cdot) \ud s\ud x\mu(\ud \xi).
  \end{split}
\end{equation}
It is clear that
\begin{equation*}
  t\in[0,\infty)\mapsto \int_{0}^{t}\int_{A}\int_B h(s,x,\xi,\cdot) M(\ud s,\ud x,\ud \xi,\cdot)\in\R
\end{equation*}
is an $\{\cF_t\}$-martingale. Moreover, for any $\{\cF_t\}$-predictable integrand $h$ satisfying
\begin{equation*}
  \E\left[ \int_{0}^{t}\int_{A}\int_B \left|h(s,x,\xi,\cdot)\right|^2 \ud s\ud x \mu(\ud\xi)\right]<\infty,
\end{equation*}
for any $t\ge 0$ and some $(A,B)\in\cB(\R^d)\times\cB(E)$, the stochastic integral defined in \eqref{e:de-sto-int} is also well defined by a limit procedure, see \cite[Section II.3]{Ikeda1981Stochastic}. And then
\begin{equation*}
  t\in[0,\infty)\mapsto \int_{0}^{t}\int_{A}\int_B h(s,x,\xi,\cdot) M(\ud s,\ud x,\ud \xi,\cdot)\in\R
\end{equation*}
is a square integrable $\{\cF_t\}$-martingale with the quadratic variation process
\begin{equation*}
  \left\langle \int_{0}^{t}\int_{A}\int_B h(s,x,\xi,\cdot) M(\ud s,\ud x,\ud \xi,\cdot) \right\rangle_t= \int_{0}^{t}\int_{A}\int_B \left|h(s,x,\xi,\cdot)\right|^2 \ud s\ud x \mu(\ud\xi).
\end{equation*}

\begin{remark}
It is clear that $M$ defined by \eqref{e:de-M} is a worthy, orthogonal, $\{\cF_t\}$-martingale measure in the context of Walsh \cite{Walsh1986Anintroduction}. Thus stochastic integrals of $\{\cF\}_t$-predictable integrands with respect to $M$ can be defined alternatively by the method in \cite[Chapter 2]{Walsh1986Anintroduction}.
\end{remark}

For the Poisson random measure $N$ and its compensating martingale measure $M$, we  can define the heuristic ``Radon–Nikodym derivatives'' in the sense of distributions
\begin{equation}\label{e:de-Ntx}
  N_{t,x}(B,\omega):=\frac{N(\ud t\ud x,B,\omega)}{\ud t\ud x}(t,x)
\end{equation}
and
\begin{equation}\label{e:de-Mtx}
  M_{t,x}(B,\omega):=\frac{M(\ud t\ud x,B,\omega)}{\ud t\ud x}(t,x)=N_{t,x}(B,\omega)-\mu(B)
\end{equation}
for $(t,x)\in[0,\infty)\times\R^d$ and $(B,\omega)\in\cB(E)\times\Omega$.  Here, $M_{t,x}$ is the \textit{Poisson space–time white noise}.

By L\'evy–It\^o decomposition, a L\'evy space-time white noise has the following structure
\begin{equation}\label{e:de-levy}
  F_{t,x}(\omega)=W_{t,x}(\omega)+\int_{U_0}h_1(t,x;\xi)M_{t,x}(\ud\xi,\omega)+\int_{E\setminus U_0}h_2(t,x;\xi)N_{t,x}(\ud\xi,\omega),
\end{equation}
for $\omega\in\Omega$, where $h_1,h_2:[0,\infty)\times\R^d\times E\mapsto\R$ are measurable and $U_0\in\cB(E)$ with $\mu(E\setminus U_0)<\infty$. Formally, $W_{t,x}:=\frac{\partial^2 W(t,x)}{\partial t\partial x}$ is a Gaussian space-time white noise on $[0,\infty)\times\R^d$ used initially by Walsh \cite{Walsh1986Anintroduction}, where $W(t,x)$ is a Brownian sheet on $[0,\infty)\times\R^d$.  $M_{t,x}$ and $N_{t,x}$ are defined formally as Radon–Nikodym derivatives as in \eqref{e:de-Mtx} and \eqref{e:de-Ntx}, respectively.

In this paper, we are also interested in studying the FSPDEs driven by pure jump L\'evy noise (i.e., without Gaussian component) in the following manner
\begin{equation}\label{e:de-jumplevy}
  F_{t,x}(\omega)=\int_{U_0}h_1(t,x;\xi)M_{t,x}(\ud\xi,\omega)+\int_{E\setminus U_0}h_2(t,x;\xi)N_{t,x}(\ud\xi,\omega).
\end{equation}

The following proposition is borrowed from \cite[Proposition 2.1]{Wu2012On}, which is derived from the $L^p$-theory for stochastic integral with respect to a Poisson random measure introduced by \cite[Theorem 8.23 (i)]{Peszat2007Stochastic}. This result is a key point for us to investigate $L^p(\R^d)$-valued solutions to the FSPDEs
\eqref{e:fde} above.

\begin{proposition}\label{prop:boundpoisson}
  Let $p\in[1,2]$. For any $\cF_t$-adapted function $\phi:[0,\infty)\times\R^d\times E\times \Omega\to\R$ satisfying
  \begin{equation*}
    \int_{0}^{t}\int_{\R^d}\int_E \E\left[\left|\phi(s,x,\xi)\right|^p\right]\ud s\ud x \mu(\ud\xi)<\infty,
  \end{equation*}
  the stochastic integral
  \begin{equation*}
    \int_{0}^{t}\int_{\R^d}\int_E \phi(s,x,\xi)   M(\ud s,\ud x,\ud\xi)
  \end{equation*}
   is well defined in $L^p(\Omega,\cF,\P)$ and is $\{\cF_t\}$-adapted. Moreover, the following inequality holds
   \begin{equation*}
     \E\left[\left|\int_{0}^{t}\int_{\R^d}\int_E \phi(s,x,\xi) M(\ud s,\ud x,\ud\xi)\right|^p\right]\le C\int_{0}^{t}\int_{\R^d}\int_E \E\left[\left|\phi(s,x,\xi)\right|^p\right]\ud s\ud x \mu(\ud\xi).
   \end{equation*}
\end{proposition}

\subsection{Fundamental solution}
In this subsection, we shall prove two lemmas for the fundamental solution to Eq. \eqref{e:fde} for later use.
The derivation of the fundamental solution was first given in \cite[Theorem 4.1]{Chen2019Nonlinear}, see also \cite[Theorem 2.8]{Chen2024Moments}.
For $\alpha\in (0,\infty)$, $\beta\in (0,2]$, and $\gamma\ge 0$, the solution to the following deterministic,
homogeneous partial differential equation {
\begin{align*}
  \begin{cases}
    \left(\partial_t^\beta + \dfrac{\nu}{2} (-\Delta)^{\alpha/2}\right) u(t,x) = 0, & t>0,\: x\in\R^d, \\[1em]
    u(0, x)=\delta(x),                                          & \text {for } \beta \in(0,1], \\
	u(0, x)=0, \quad \dfrac{\partial}{\partial t} u(0,x)=\delta(x),                 & \text {for } \beta \in(1,2),
  \end{cases}
\end{align*}
is the fundamental solution to Eq. \eqref{e:fde},} which is denoted by
$Z_{\alpha,\beta,d}(t,x)=:Z(t,x)$ for clarity, and  the solution of the following deterministic partial differential equation
\begin{align*}
  \begin{cases}
    \left(\partial_t^\beta + \dfrac{\nu}{2} (-\Delta)^{\alpha/2}\right) u(t,x) = I_{t}^\gamma\left[f(t,x)\right], & \qquad t>0,\: x\in\R^d, \\[1em]
    \left.\dfrac{\partial^k}{\partial t^k} u(t,x)\right|_{t = 0} = u_k(x), & \qquad 0\le k\le \Ceil{\beta}-1, \:\: x\in\R^d,
  \end{cases}
\end{align*}
can be then expressed by
\begin{align}\label{E:Duhamel}
  u(t,x) = J_0(t,x) + \int_0^t \ud s \int_{\R^d} \ud y\: f(s,y) \: D_{t}^{\Ceil{\beta}-\beta-\gamma} Z(t-s,x-y),
\end{align}
where the symbol $D_{t}^{\Ceil{\beta}-\beta-\gamma}$ denotes the \textit{Riemann-Liouville derivative} of order $\Ceil{\beta}-\beta-\gamma$ acting on the time variable, namely
\begin{align*}
  D_{t}^{\Ceil{\beta}-\beta-\gamma}f(t)&:=\frac{\ud^n}{\ud t^n} \left[I_{t}^{n-\Ceil{\beta}-\beta-\gamma}f (t)\right], \quad n = \Ceil{\Ceil{\beta}-\beta-\gamma},
\end{align*}
and
\begin{align}\label{E:J0}
  J_0(t,x):= \sum_{k = 0}^{\Ceil{\beta}-1}\int_{\R^d} u_{k}(y) \partial^{\Ceil{\beta}-1-k} Z(t,x-y) \ud y
\end{align}
is the solution to the following
\begin{align*}
  \begin{cases}
    \left(\partial_t^\beta + \dfrac{\nu}{2} (-\Delta)^{\alpha/2}\right) u(t,x) = 0, & \qquad t>0,\: x\in\R^d, \\[1em]
    \left.\dfrac{\partial^k}{\partial t^k} u(t,x)\right|_{t = 0} = u_k(x), & \qquad 0\le k\le \Ceil{\beta}-1, \:\: x\in\R^d.
  \end{cases}
\end{align*}
We further denote
\begin{align*}
  &Y(t,x) := Y_{\alpha,\beta,\gamma,d}(t,x) = D_{t}^{\Ceil{\beta}-\beta-\gamma}Z_{\alpha,\beta,d}(t,x), \\
  &Z^* (t,x) := Z^{*}_{\alpha,\beta,d}(t,x) = \frac{\partial}{\partial t} Z_{\alpha,\beta,d}(t,x),
\end{align*}
then, for $\beta\in (0,2)$, one has the following explicit expressions:
\begin{align*}
   Z(t,x) = \pi^{-\frac d2} t^{\Ceil{\beta}-1} |x|^{-d}
   \FoxH{2,1}{2,3}{\frac{|x|^\alpha}{2^{\alpha-1}\nu t^\beta}}{(1,1),\:(\Ceil{\beta},\beta)}
   {(d/2,\alpha/2),\:(1,1),\:(1,\alpha/2)},
\end{align*}
\begin{align}\label{E:Yab}
    Y(t,x) = \pi^{-\frac d2} |x|^{-d}t^{\beta+\gamma-1}
     \FoxH{2,1}{2,3}{\frac{|x|^\alpha}{2^{\alpha-1}\nu t^\beta}}
     {(1,1),\:(\beta+\gamma,\beta)}{(d/2,\alpha/2),\:(1,1),\:(1,\alpha/2)},
\end{align}
and, for $\beta\in (1,2)$,
\begin{align*}
  Z^{*}(t,x)=
  \pi^{-\frac d2} |x|^{-d}
  \FoxH{2,1}{2,3}{\frac{|x|^\alpha}{2^{\alpha-1}\nu t^\beta}}{(1,1),\:(1,\beta)}
  {(d/2,\alpha/2),\:(1,1),\:(1,\alpha/2)},
\end{align*}
where $\FoxH{2,1}{2,3}{\cdots}{\cdots}{\cdots}$ refers to the Fox $H$-function, see Kilbas \cite{Kilbas2004H}. We also have the following Fourier transforms:
\begin{align}\label{E:FZ}
 \cF Z(\xi)        & = t^{\Ceil{\beta}-1} E_{\beta,\Ceil{\beta}}(-\tfrac12 \nu t^\beta |\xi|^\alpha),              \\
 \cF Y(t,\cdot)(\xi) & = t^{\beta+\gamma-1} E_{\beta,\beta+\gamma}(-\tfrac12 \nu t^\beta |\xi|^\alpha), \notag\\
 \cF Z^*(t,\cdot)(\xi)      & = t^{k} E_{\beta,k+1}(-\tfrac12\nu t^\beta |\xi|^\alpha),
 \label{E:FZ*}
\end{align}
where $E_{a, b}(z)$ is the two-parameter Mittag-Leffler function defined by
\begin{equation*}
  E_{a, b}(z):=\sum_{k=0}^{\infty}\frac{z^k}{\Gamma(a k+b)},
\end{equation*}
for all $a>0, b\in\mathbb{C}$ (see, e.g., \cite[Sect.1.2]{Podlubny1999Fractional}).

Throughout the rest of the article, we will follow the convention to write
\begin{equation}\label{e:de-p}
  p(t,x):=Y(t,x)=\pi^{-\frac d2} |x|^{-d}t^{\beta+\gamma-1} \FoxH{2,1}{2,3}{\frac{|x|^\alpha}{2^{\alpha-1}\nu t^\beta}}
     {(1,1),\:(\beta+\gamma,\beta)}{(d/2,\alpha/2),\:(1,1),\:(1,\alpha/2)}.
\end{equation}

We are now in the position to show the following two results, which will be used in the sequel.
\begin{lemma}\label{le:upb-fmsl}
Let $p>0$ if $\alpha\ge d$ or $0<p<\frac{d}{d-\alpha}$ if $\alpha<d$. Then,
there exists a constant $C$ depending on $(\alpha,\beta,\gamma,d,p)$ such that for all $t>0$,
  \begin{equation*}
    \sup_{y\in\R^d}\int_{\R^d} \left|p(t,x-y)\right|^p\ud x= \int_{\R^d} \left|p(t,x)\right|^p\ud x = C\: t^{p(\beta+\gamma-1)+\frac{\beta d}{\alpha}(1-p)}.
  \end{equation*}
\end{lemma}
\begin{proof}
Under our assumption, we have first of all that
  \begin{equation}\label{e-con-fs}
   p\alpha>d(p-1).
  \end{equation}
  Using the change of variable $x=x'+y$, for all $y\in\R^d$, we have
  \begin{align*}
     &\int_{\R^d} \left|p(t,x-y)\right|^p\ud x= \int_{\R^d} \left|p(t,x)\right|^p\ud x\\
     &=t^{p(\beta+\gamma-1)} \int_{\R^d} \left|\pi^{-\frac d2} |x|^{-d}\FoxH{2,1}{2,3}{\frac{|x|^\alpha}{2^{\alpha-1}\nu t^\beta}}
     {(1,1),\:(\beta+\gamma,\beta)}{(d/2,\alpha/2),\:(1,1),\:(1,\alpha/2)} \right|^p\ud x.
  \end{align*}
  Then, by the change of variable $x=t^{\beta/\alpha}x'$, we have
  \begin{align}\label{e:int-p1x}
     & \int_{\R^d} \left|p(t,x)\right|^p\ud x\notag\\
     &=t^{p(\beta+\gamma-1)+\frac{\beta d}{\alpha}(1-p)}\int_{\R^d} \left|\pi^{-\frac d2} |x|^{-d} \FoxH{2,1}{2,3}{\frac{|x|^\alpha}{2^{\alpha-1}\nu }}
     {(1,1),\:(\beta+\gamma,\beta)}{(d/2,\alpha/2),\:(1,1),\:(1,\alpha/2)}\right|^p\ud x.
  \end{align}
  Using the asymptotic property of $p(1,x)$ \cite[Lemma 4.3 and Lemma 4.5]{Chen2019Nonlinear}, we have
  \begin{equation*}
    |p(1,x)| \le C\big( 1+|x|^{\alpha-d}+|\log (|x|)|\big),
  \end{equation*}
 as $|x|\to0$  and
  \begin{equation*}
    |p(1,x)| \le
\begin{cases}
 C |x|^{-(d+\alpha)}& \text{if $\alpha\ne 2$},\\
 C |x|^{a}e^{-b|x|^c}
 &\text{if $\alpha=2$,}
\end{cases}
  \end{equation*}
  as $|x|\to\infty$, where $a$ is a nonnegative and $b,c$ are strictly positive. Then, we can conclude that the integral in \eqref{e:int-p1x} is finite under the condition \eqref{e-con-fs}. The proof is complete.
\end{proof}

\begin{lemma}\label{le:upb-fms-deri}
Let $p>0$ if $\alpha\ge d+1$ or $0<p<\frac{d}{d+1-\alpha}$ if $\alpha<d$. Then,
there exists a constant $C$ depending on $(\alpha,\beta,\gamma,d,p)$ such that for all $t>0$,
  \begin{equation*}
    \int_{\R^d} \left|\frac{\partial }{\partial x_i} p(t,x)\right|^p\ud x = C\: t^{p(\beta+\gamma-1)+\frac{\beta}{\alpha}(d-pd-p)}.
  \end{equation*}
\end{lemma}
\begin{proof}
First, we have
  \begin{equation}\label{e-con-fs-partial}
    p \alpha> d(p-1)+p.
  \end{equation}
By \cite[(2.2.1)]{Kilbas2004H}, we have
\begin{align}\label{e:udxi-ptx}
   & \frac{\ud}{\ud x}\left[x^{-d}\times \FoxH{2,1}{2,3}{C x^\alpha}
       {(1,1),\:(\beta+\gamma,\beta)}{(d/2,\alpha/2),\:(1,1),\:(1,\alpha/2)}\right]\notag\\
       &= \, x^{-d-1} \FoxH{2,2}{3,4}{C x^\alpha}
       {(d,\alpha),\:(1,1),\:(\beta+\gamma,\beta)}{(d/2,\alpha/2),\:(1,1),\:(1,\alpha/2),\:(d+1,\alpha)}.
\end{align}
Then, by \eqref{e:de-p} and \eqref{e:udxi-ptx}, for $i=1, \dots, d$,
\begin{align*}
  &\frac{\partial}{\partial x_i}p(t,x)= \pi^{-d/2}|x|^{-d-1} t^{\beta+\gamma-1}\times \FoxH{2,2}{3,4}{\frac{ |x|^\alpha}{2^{\alpha-1}\nu t^\beta }}
       {(d,\alpha),\:(1,1),\:(\beta+\gamma,\beta)}{(d/2,\alpha/2),\:(1,1),\:(1,\alpha/2),\:(d+1,\alpha)}\times\frac{x_i}{|x|}.
\end{align*}
Using the change of variable $x=t^{\beta/\alpha}x'$, we have
\begin{align*}
  &\int_{\R^d} \bigg|\frac{\partial}{\partial x_i}p(t,x)\bigg|^p \ud x \notag\\
  &=t^{p(\beta+\gamma-1)} \int_{\R^d}  \bigg| \pi^{-d/2}|x|^{-d-1}  \FoxH{2,2}{3,4}{\frac{ |x|^\alpha}{2^{\alpha-1}\nu  t^\beta }}
       {(d,\alpha),\:(1,1),\:(\beta+\gamma,\beta)}{(d/2,\alpha/2),\:(1,1),\:(1,\alpha/2),\:(d+1,\alpha)} \frac{x_i}{|x|}\bigg|^p \ud x\notag\\
  &=t^{p(\beta+\gamma-1)+\frac{\beta}{\alpha}(d-pd-p)} \int_{\R^d}  \bigg| \pi^{-d/2}|x|^{-d-1}  \FoxH{2,2}{3,4}{\frac{ |x|^\alpha}{2^{\alpha-1}\nu  }}
       {(d,\alpha),\:(1,1),\:(\beta+\gamma,\beta)}{(d/2,\alpha/2),\:(1,1),\:(1,\alpha/2),\:(d+1,\alpha)} \frac{x_i}{|x|}\bigg|^p \ud x.
\end{align*}
Applying spherical coordinate system to see that
\begin{align}\label{e:int-H2234x}
  &\int_{\R^d} \bigg|\frac{\partial}{\partial x_i}p(t,x)\bigg|^p \ud x \notag\\
  &=C\: t^{p(\beta+\gamma-1)+\frac{\beta}{\alpha}(d-pd-p)} \int_{0}^\infty  \bigg| r^{-d-1}  \FoxH{2,2}{3,4}{\frac{ r^\alpha}{2^{\alpha-1}\nu }}
       {(d,\alpha),\:(1,1),\:(\beta+\gamma,\beta)}{(d/2,\alpha/2),\:(1,1),\:(1,\alpha/2),\:(d+1,\alpha)}\bigg|^p \ud r.
\end{align}
Following \cite[Lemma A.4]{guo2024sample}, we obtain that
\begin{equation*}
  \FoxH{2,2}{3,4}{\frac{ r^\alpha}{2^{\alpha-1}\nu}}
       {(d,\alpha),\:(1,1),\:(\beta+\gamma,\beta)}{(d/2,\alpha/2),\:(1,1),\:(1,\alpha/2),\:(d+1,\alpha)}\le C \left( r^{d+2} +r^{\alpha}|\log(r)| \right)
\end{equation*}
as $r\to0$,  and
\begin{equation*}
  \FoxH{2,2}{3,4}{\frac{ r^\alpha}{2^{\alpha-1}\nu }}
       {(d,\alpha),\:(1,1),\:(\beta+\gamma,\beta)}{(d/2,\alpha/2),\:(1,1),\:(1,\alpha/2),\:(d+1,\alpha)}\le C
\end{equation*}
as $r\to\infty$. Then, we can conclude that the integral in \eqref{e:int-H2234x} is finite under the condition \eqref{e-con-fs-partial}. The proof is complete.
\end{proof}

\section{Main results}\label{se:main-result}

We first present a mild formulation of Eq. \eqref{e:fde}. Using \eqref{e:de-levy} and \eqref{E:Duhamel}, we have
\begin{align*}
       u(t,x)= & J_0(t,x)+\int_{0}^{t}\int_{\R^d}p(t-s,x-y)f(s,y,u(s,y))\ud s\ud y \\
   &- \sum_{j=1}^{d}\int_{0}^{t} \int_{\R^d}  p(t-s,x-y) \left[\frac{\partial}{\partial y_j } q_j(s,y,u(s, y))\right] \ud s\ud y \\
   &+\int_{0}^{t}\int_{\R^d}p(t-s,x-y) F(\ud s,\ud y). \\
\end{align*}
Apply integration by parts formula then yields the following
\begin{equation}\label{e:de-so-wal}
  \begin{split}
       u(t,x)= &J_0(t,x)+\int_{0}^{t}\int_{\R^d}p(t-s,x-y)f(s,y,u(s,y))\ud s\ud y  \\
  &+\sum_{j=1}^{d} \int_{0}^{t} \int_{\R^d} \left[\frac{\partial}{\partial y_j } p(t-s,x-y)\right]  q_j(s,y,u(s, y,\omega))\ud s\ud y  \\
  &+ \int_{0}^{t}\int_{\R^d} p(t-s,x-y)\sigma(s,y,u(s,y)) W(\ud s,\ud y)  \\
  &+\int_{0}^{t}\int_{\R^d}\int_{U_0} p(t-s,x-y)\sigma(s,y,u(s,y)) h_1(s,y;\xi) M(\ud s,\ud y,\ud\xi)  \\
  &+ \int_{0}^{t}\int_{\R^d}\int_{E\setminus U_0} p(t-s,x-y)\sigma(s,y,u(s,y)) h_2(s,y;\xi) N(\ud s,\ud y,\ud\xi).
  \end{split}
\end{equation}
Following \cite{Walsh1986Anintroduction}, here and in the sequel, an $\{\cF_t\}$-predictable version of $u$ is taken as $u$ which is c\`adl\`ag in time variable $t\in[0,\infty)$, in order that the stochastic integrals with respect to $M$ and $N$ are respectively well defined.

Next, we reformulate Eq. \eqref{e:de-so-wal} by the following consideration. Observing that $\mu(E\setminus U_0)<\infty$, we have
\begin{align*}
   &\int_{0}^{t}\int_{\R^d}\int_{E\setminus U_0} p(t-s,x-y)\sigma(s,y,u(s,y)) h_2(s,y;\xi)N(\ud s,\ud y,\ud\xi)\\
   &= \int_{0}^{t}\int_{\R^d}\int_{E\setminus U_0} p(t-s,x-y)\sigma(s,y,u(s,y)) h_2(s,y;\xi)M(\ud s,\ud y,\ud\xi)\\
   &\quad + \int_{0}^{t}\int_{\R^d}\int_{E\setminus U_0} p(t-s,x-y)\sigma(s,y,u(s,y)) h_2(s,y;\xi)\ud s \ud y \mu(\ud\xi).
\end{align*}
Thus, without loss of generality, we shall consider the equation in the following form:
\begin{equation}\label{e:de-solution}
  \begin{split}
     u(t,x)= & J_0(t,x)+\int_{0}^{t}\int_{\R^d}p(t-s,x-y)f(s,y,u(s,y))\ud s\ud y\\
     &+\sum_{j=1}^{d} \int_{0}^{t} \int_{\R^d} \left[\frac{\partial}{\partial y_j } p(t-s,x-y)\right]  q_j(s,y,u(s, y,\omega))\ud s\ud y \\
     &+ \int_{0}^{t}\int_{\R^d} p(t-s,x-y)\sigma(s,y,u(s,y)) W(\ud s,\ud y)\\
   &+ \int_{0}^{t}\int_{\R^d}\int_{E} p(t-s,x-y) h(s,y,u(s,y);\xi) M(\ud s,\ud y,\ud\xi),
  \end{split}
\end{equation}
where $f,\sigma:[0,\infty)\times\R^d\times\R\mapsto\R$ and $h:[0,\infty)\times\R^d\times\R\times E\mapsto\R $ are measurable.
Clearly, Eq. \eqref{e:de-solution} is a mild formulation of the following (formal) equation
\begin{equation}\label{e:fde-white}
  \begin{split}
     &\left(\partial_t^{\beta}+\frac{\nu}{2}\left(-\Delta\right)^{\alpha / 2}\right) u(t, x)=I_{t}^{\gamma}\bigg[ f(t,x,u(t, x))\\
  &\qquad -\sum_{j=1}^{d} \frac{\partial}{\partial x_j} q_j(t,x,u(t, x,\omega))+ \sigma(t,x,u(t, x)) W_{t,x}+ \int_{E}h(t,x,u(t, x);\xi)M_{t,x}(\ud\xi)  \bigg].
  \end{split}
\end{equation}
Similarly, for the pure jump L\'evy noise, we shall consider
\begin{equation}\label{e:de-solution2}
  \begin{split}
     u(t,x)= & J_0(t,x)+\int_{0}^{t}\int_{\R^d}p(t-s,x-y)f(s,y,u(s,y))\ud s\ud y\\
     &+\sum_{j=1}^{d} \int_{0}^{t} \int_{\R^d} \left[\frac{\partial}{\partial y_j } p(t-s,x-y)\right]  q_j(s,y,u(s, y,\omega))\ud s\ud y \\
   &+ \int_{0}^{t}\int_{\R^d}\int_{E} p(t-s,x-y) h(s,y,u(s,y);\xi) M(\ud s,\ud y,\ud\xi),
  \end{split}
\end{equation}
which is a mild formulation of the following (formal) equation
\begin{align}\label{e:fde-jump}
 & \left(\partial_t^{\beta}+\frac{\nu}{2}\left(-\Delta\right)^{\alpha / 2}\right) u(t, x)\notag\\
  &\qquad=I_{t}^{\gamma} \bigg[ f(t,x,u(t, x))-\sum_{j=1}^{d} \frac{\partial}{\partial x_j} q_j(t,x,u(t, x,\omega))+ \int_{E}h(t,x,u(t, x);\xi)M_{t,x}(\ud\xi)  \bigg].
\end{align}

 Now  we give the precise definitions of the global and local solutions.

 \begin{definition}[Global solution]
   We say that an $\{\cF_t\}$-adapted stochastic process $u(t,x):[0,T]\times\R^d\mapsto\R$ is a global solution of \eqref{e:fde-white} for L\'evy space-time white nosie (resp. \eqref{e:fde-jump} for pure jump L\'evy noise), if for any $T>0$, $\{u(t,x)\}_{t\in[0,T]}$ satisfies the integral equation \eqref{e:de-solution} (resp. \eqref{e:de-solution2}) for each $t\in[0,T]$. Furthermore, we say that the solution is unique, if $u$ and $v$ are any two global solutions, then for all $(t,x)\in[0,T]\times\R^d$, $u(t,x)=v(t,x)$ almost surely.
 \end{definition}

 \begin{definition}[Local solution]
We  call an $\{\cF_t\}$-adapted stochastic process $u(t,x):[0,T]\times\R^d\to\R$ which is c\`adl\`ag in time variable $t\in[0,T]$ a local solution of \eqref{e:fde-white} for L\'evy space-time white nosie (resp. \eqref{e:fde-jump} for pure jump L\'evy noise), if there exists an $\{\cF_t\}$-stopping time $\tau:\Omega\to(0,T]$ such that $\{u(t,x)\}_{t\le\tau}$ satisfies the integral equation \eqref{e:de-solution} (resp. \eqref{e:de-solution2}). Moreover, we say the uniqueness holds for local solution if there exists another local solution $\tilde{u}$ with a stopping time $\tilde{\tau}$, then  we have $u(t,x)=\tilde{u}(t,x)$ almost surely, for all $(t,x)\in[0,\tau\wedge\tilde{\tau}]\times\R^d$.
 \end{definition}

We are ready to state the main results of this paper.

\begin{theorem}[L\'evy space-time white noise]\label{th:1}
Let the parameters $(\alpha,\beta,\gamma,d)$ satisfy the inequality
\begin{equation}\label{e:con-th1}
  2\alpha+\min\left\{\frac{\alpha}{\beta}(2\gamma-1),-2 \right\}>d.
\end{equation}
Assume that there exist positive functions $g_1,g_2,g_3\in L^1(\R^d)$ such that the following growth conditions:
\begin{equation*}
  |f(t,x,z)|^2 +\sum_{j=1}^{d} |q_j(t,x,z)|^2 \le g_1(x)+ C|z|^2,
\end{equation*}
\begin{equation}\label{e:con-sigma-h}
  |\sigma(t,x,z)|^2+\int_{\E}|h(t,x,z;\xi)|^2\mu(\ud\xi)  \le g_2(x)+ C|z|^2,
\end{equation}
and Lipschitz conditions
\begin{align}\label{e:con-f2}
   &|f(t,x,z_1)-f(t,x,z_2)|^2+\sum_{j=1}^{d} |q_j(t,x,z_1)-q_j(t,x,z_2)|^2 \le \left[g_3(x)+C(|z_1|^2+|z_2|^2)\right]|z_1-z_2|^2,
\end{align}
\begin{equation*}
  |\sigma(t,x,z_1)-\sigma(t,x,z_2)|^2+\int_{\E}|h(t,x,z_1;\xi)-h(t,x,z_2;\xi)|^2\mu(\ud\xi)  \le C|z_1-z_2|^2
\end{equation*}
hold for all $(t,x)\in[0,\infty)\times\R^d$ and $z_1,z_2,z_3\in\R$. Then for every $\cF_0$-measurable $u_0,u_1:\R^d\times\Omega\to\R$ with
\begin{equation*}
\begin{cases}
 \displaystyle \E\left[\int_{\R^d}|u_0(x,\cdot)|^2\ud x\right]<\infty, & \mbox{if } \beta\in(0,1], \vspace{0.2cm}\\
  \displaystyle \E\left[\int_{\R^d}|u_0(x,\cdot)|^2\ud x\right]+ \E\left[\int_{\R^d}|u_1(x,\cdot)|^2\ud x\right]<\infty , & \mbox{if } \beta\in(0,2).
\end{cases}
\end{equation*}
there exist a unique local solution to \eqref{e:de-solution} with the following property
\begin{equation*}
  \E \left[\Vert u(t\wedge\tau,\cdot)\Vert_2^2\right]<\infty,
\end{equation*}
for all $t\in[0,T]$.  Furthermore, there exists a predictable modification for the solution $u(t)$ of \eqref{e:de-solution}.
\end{theorem}

\begin{theorem}[Pure jump L\'evy noise]\label{th:purejump}
Let $p\in[1,2]$ and  the parameters $(\alpha,\beta,\gamma,d,p)$ satisfy the inequality
  \begin{equation}\label{e:con-mainth2}
    p\alpha+\min\left\{\frac{\alpha}{\beta}(p\gamma-p+1),-p \right\}>d(p-1).
  \end{equation}
  Assume that there exist positive functions $a_1\in L^p(\R^d),a_2\in L^1(\R^d),a_3\in L^p(\R^d)$ such that the following growth conditions:
\begin{equation}\label{e:con-f-p}
  |f(t,x,z)|+\sum_{j=1}^{d} |q_j(t,x,z)| \le a_1(x)+ C|z|,
\end{equation}
\begin{equation}\label{e:con-sigma-h-p}
  \int_{\E}|h(t,x,z;\xi)|^p\mu(\ud\xi)  \le a_2(x)+ C|z|^p,
\end{equation}
and Lipschitz conditions
\begin{align}\label{e:con-f2-p}
   &|f(t,x,z_1)-f(t,x,z_2)|+\sum_{j=1}^{d} |q_j(t,x,z_1)-q_j(t,x,z_2)| \le \big[a_3(x)+C(|z_1|+|z_2|)\big] |z_1-z_2|,
\end{align}
\begin{equation}\label{e:con-h2-p}
  \int_{\E}|h(t,x,z_1;\xi)-h(t,x,z_2;\xi)|^p\mu(\ud\xi)  \le C|z_1-z_2|^p,
\end{equation}
hold for all $(t,x)\in[0,T]\times\R^d$ and $z_1,z_2,z_3\in\R$. Let $\cF_0$-measurable $u_0,u_1:\R^d\times\Omega\to\R$ satisfy
\begin{equation}\label{e:initial-p}
\begin{cases}
  \displaystyle \E\left[\int_{\R^d}|u_0(x,\cdot)|^p\ud x\right]<\infty, & \mbox{if } \beta\in(0,1] \vspace{0.2cm}\\
  \displaystyle \E\left[\int_{\R^d}|u_0(x,\cdot)|^p\ud x\right]+ \E\left[\int_{\R^d}|u_1(x,\cdot)|^p\ud x\right]<\infty , & \mbox{if } \beta\in(0,2).
\end{cases}
\end{equation}
Then  there exist a unique local solution to \eqref{e:de-solution2} with the following property
\begin{equation*}
  \E \left[\Vert u(t\wedge\tau,\cdot)\Vert_p^p\right]<\infty,
\end{equation*}
for all $t\in[0,T]$.  Furthermore, there exists a predictable modification for the solution $u(t)$ of \eqref{e:de-solution2}.
\end{theorem}

\begin{remark}\label{re:comp-white}
   Let us compare the above theorem with some known results:
   \begin{enumerate}
     \item For the classical stochastic Burgers equation (i.e., $\alpha=2$, $\beta=1$ and $\gamma=0$) with L\'evy white noise, condition \eqref{e:con-th1} is equivalent to
     \begin{equation*}
       d<2,
     \end{equation*}
      which agrees with Theorem 3.1 in \cite{Truman2003Stochastic}.
      \item For the fractional stochastic Burgers equation (i.e., $\alpha\in(0,2]$, $\beta=1$ and $\gamma=0$) with L\'evy white noise, condition \eqref{e:con-th1} is equivalent to
     \begin{equation*}
       \alpha>\frac{d+2}{2},
     \end{equation*}
      which coincides with Theorem 3.1 in \cite{Truman2005Fractal}, see also \cite[Theorem 3.1]{Truman2006On} and \cite[Theorem 2.3]{Jacob2010Solving}.
      \item For the fractional stochastic Burgers equation (i.e., $\alpha\in(0,2]$, $\beta=1$ and $\gamma=0$) with pure jump noise, condition \eqref{e:con-mainth2} is equivalent to
          \begin{equation*}
            p\alpha-p>d(p-1),
          \end{equation*}
          which coincides with Theorem 3.1 in \cite{Wu2012On}.
   \end{enumerate}
\end{remark}

Theorems \ref{th:1} and \ref{th:purejump} only investigate the existence and uniqueness of the local solutions. The difficulties for the global solutions mainly come from the polynomial nonlinearity on $f$ and $q_i$, see \eqref{e:con-f2} and \eqref{e:con-f2-p}. If we strengthen the conditions \eqref{e:con-f2} and \eqref{e:con-f2-p}, we can obtain the existence and uniqueness of the global solutions ‌as following theorem shows.

\begin{theorem}[Global solution]\label{th:global}
\begin{enumerate}
  \item If we change the condition \eqref{e:con-f2} to
  \begin{equation*}
    |f(t,x,z_1)-f(t,x,z_2)|^2+\sum_{j=1}^{d} |q_j(t,x,z_1)-q_j(t,x,z_2)|^2 \le C|z_1-z_2|^2,
  \end{equation*}
  and assume that the remaining conditions in Theorem \ref{th:1} hold, then there exists a unique global solution to \eqref{e:de-solution} with the following property
\begin{equation*}
  \E \left[\Vert u(t,\cdot)\Vert_2^2\right]<\infty,
\end{equation*}
for all $t\in[0,T]$.

  \item If we change the condition \eqref{e:con-f2-p} to
  \begin{equation*}
    |f(t,x,z_1)-f(t,x,z_2)|+\sum_{j=1}^{d} |q_j(t,x,z_1)-q_j(t,x,z_2)| \le C |z_1-z_2|,
  \end{equation*}
  and assume that the remaining conditions in Theorem \ref{th:purejump} hold, then there exists a unique global solution to \eqref{e:de-solution2} with the following property
\begin{equation*}
  \E \left[\Vert u(t,\cdot)\Vert_p^p\right]<\infty,
\end{equation*}
for all $t\in[0,T]$.
\end{enumerate}
\end{theorem}

\begin{remark}\label{re:comp-global}
  According to the Theorem \ref{th:global}, for the fractional stochastic Burgers equation (i.e., $\alpha\in(0,2]$, $\beta=1$ and $\gamma=0$) with pure jump noise, there exists a global solution if
  \begin{equation*}
    p\alpha-p>d(p-1).
  \end{equation*}
  This condition differs from that in \cite[Corollary 3.2]{Wu2012On} where the nonlinearity term (i.e., $q_i$ in \eqref{e:fde}) is assumed to be zero, thereby resulting in a more relaxed condition.
\end{remark}

\section{Proofs of main results}\label{se:proof}

In this section, we will prove Theorems \ref{th:1} and \ref{th:purejump}. Since Theorem \ref{th:global} been easily verified by the similar arguments in the proofs of Theorems \ref{th:purejump} and \ref{th:purejump}, so we omit the details. Now, we first show Theorem \ref{th:purejump}.

\begin{proof}[Proof of Theorem \ref{th:purejump}] Due to the use of Young’s convolution inequality and H\"older's inequality, we first deal with the case $p\in(1,2]$. To this end, our proof is divided into the following four steps.

{\em Step 1.}
For $p\in[1,2]$ and for arbitrarily fixed $T>0$, let us denote by $B_{T,p}$ the collection of all $\cF_t$-adapted functions $u:[0,T]\times\R^d\times\Omega\mapsto u(t,\cdot,\omega)\in L^p(\R^d)$ satisfying
\begin{equation}\label{e:de-norm1}
  \Vert u\Vert_{T,p}:=\sup_{t\in[0,T]} \E\left[\Vert u(t,\cdot) \Vert^p_p\right]^{1/p}<\infty.
\end{equation}
It is clear that $\Vert \cdot\Vert_{T,p}$ is a norm and $B_{T,p}$ endowed with $\Vert \cdot\Vert_{T,p}$ is a Banach space.

For each $n\in\N$, let $\pi_n$ be the mapping from $L^p(\R^d)$ to $L^p(\R^d)$, i.e.,
\begin{equation*}
  \pi_n u(t,x)=
  \begin{cases}
    u(t,x), & \mbox{if } \Vert u\Vert_p \le n \\
    \frac{n}{\Vert u(t,\cdot)\Vert_p} u(t,x), & \mbox{if } \Vert u\Vert_p \ge n.
  \end{cases}
\end{equation*}
Clearly, we have $\pi_n u\le n$ and the mapping $\pi_n:L^p(\R^d)\to L^p(\R^d)$ is Lipschitz continuous, i.e.,
\begin{equation}\label{e:pin-contract}
  \left\Vert\pi_n u-\pi_n v\right\Vert_p \le \left\Vert u-v\right\Vert_p,\quad \forall u,v\in L^p(\R^d),
\end{equation}
which implies that $\pi_n$ is also a contraction. Note that the pure jump L\'evy noise does not contain Gaussian term in \eqref{e:de-solution2}. For any fixed $n$, the truncated stochastic integral equation associated with $\pi_n$ is
\begin{equation}\label{e:de-truncatedSIE}
\begin{split}
   u(t,x)= & J_0(t,x)+ \int_{0}^{t}\int_{\R^d} p(t-s,x-y)f(s,y,\pi_n u(s,y)) \ud s \ud y\\
   &+\sum_{j=1}^{d} \int_{0}^{t} \int_{\R^d} \left[\frac{\partial}{\partial y_j } p(t-s,x-y)\right]  q_j(s,y,\pi_n u(s, y))\ud s\ud y \\
   &+ \int_{0}^{t}\int_{\R^d}\int_{E} p(t-s,x-y) h(s,y,\pi_n u(s,y);\xi)M(\ud s,\ud y,\ud\xi).
\end{split}
\end{equation}

Now we define a mapping
$\mathcal T$ on $B_{T,p}$:
\begin{equation}\label{e:de-cT}
  (\mathcal T u)(t,x) := J_0(t,x) + \sum_{i=1}^{3}(\mathcal T_i u)(t,x),
\end{equation}
with
\begin{align*}
  (\mathcal T_1 u)(t,x) & =  \int_{0}^{t}\int_{\R^d} p(t-s,x-y)f(s,y,\pi_n u(s,y)) \ud s \ud y, \\
  (\mathcal T_2 u)(t,x) & =  \sum_{j=1}^{d} \int_{0}^{t} \int_{\R^d} \left[\frac{\partial}{\partial y_j } p(t-s,x-y)\right]  q_j(s,y,\pi_n u(s, y))\ud s\ud y, \\
  (\mathcal T_3 u)(t,x) & = \int_{0}^{t}\int_{\R^d}\int_{E} p(t-s,x-y) h(s,y,\pi_n u(s,y);\xi)M(\ud s,\ud y,\ud\xi).
\end{align*}
We will show that the operator $\cT$ maps $B_{T,p}$ into itself in this step.

By \eqref{E:J0} and triangle inequality, when $\beta\in(1,2)$, we have
\begin{align*}
    \Vert J_0(t,\cdot)\Vert_p &=  \left\Vert \int_{\R^d} u_0(y)Z^*(t,\cdot-y)\ud y +\int_{\R^d} u_1(y)Z(t,\cdot-y) \ud y \right\Vert_p \\
    &\le \left\Vert \int_{\R^d} u_0(y)Z^*(t,\cdot-y) \ud y \right\Vert_p  +\left\Vert \int_{\R^d} u_1(y)Z(t,\cdot-y) \ud y \right\Vert_p.
\end{align*}
Then, apply Young's convolution inequality to obtain that
\begin{align*}
    \Vert J_0(t,\cdot)\Vert_p &\le \left\Vert u_0(\cdot) \right\Vert_p \int_{\R^d}  Z^*(t,x) \ud x   +\left\Vert u_1(\cdot)\right\Vert_p \int_{\R^d} Z(t,x) \ud x \\
   & = \left\Vert u_0(\cdot) \right\Vert_p \cF Z^*(t,\cdot)(0)   +\left\Vert u_1(\cdot)\right\Vert_p  \cF Z(t,\cdot)(0) \\
   &= \left\Vert u_0(\cdot) \right\Vert_p+ t \left\Vert u_1(\cdot)\right\Vert_p,
\end{align*}
where we have used \eqref{E:FZ} and \eqref{E:FZ*} in the last equality. The case $\beta\in(0,1]$ is similar and easier, so we can obtain that, for all $\beta\in(0,2)$
\begin{equation}\label{e:esti-j0}
  \Vert J_0(t,\cdot)\Vert_p<\infty,
\end{equation}
if the condition \eqref{e:initial-p} holds.

From \eqref{e:con-f-p} and Minkowski inequality, we get
\begin{align*}
    \Vert (\mathcal T_1 u)(t,\cdot)\Vert_p &=\left\Vert \int_{0}^{t}\int_{\R^d} p(t-s,\cdot-y)f(s,y,\pi_n u(s,y)) \ud s \ud y \right\Vert_p\\
   &\le C \left\Vert \int_{0}^{t}\int_{\R^d} |p(t-s,\cdot-y)|(a_1(y)+ |\pi_n u(s,y)|) \ud s \ud y \right\Vert_p \\
   &\le C  \int_{0}^{t} \left\Vert \int_{\R^d} |p(t-s,\cdot-y)|(a_1(y)+ |\pi_n u(s,y)|)  \ud y \right\Vert_p \ud s.
\end{align*}
Then, using Young's convolution inequality and the above estimate, we have
\begin{align*}
  \Vert (\mathcal T_1 u)(t,\cdot)\Vert_p & \le C  \int_{0}^{t} \left(\int_{\R^d} |p(t-s,x)|\ud x \right)  \left\Vert a_1(\cdot)+ |\pi_n u(s,\cdot)|\right\Vert_p \ud s.
\end{align*}
Since \eqref{e:con-mainth2} implies \eqref{e-con-fs}, we can apply Lemma \ref{le:upb-fmsl} to get
\begin{align}\label{e:j1-infty}
  \Vert (\mathcal T_1 u)(t,\cdot)\Vert_p & \le  C  \int_{0}^{t} (t-s)^{\beta+\gamma-1}  \ud s \times(\left\Vert a_1(\cdot)\right\Vert_p+n)<\infty,
\end{align}
due to $\beta+\gamma-1>-1$.

Next, using Minkowski inequality twice, we have
\begin{align*}
  \Vert (\mathcal T_2 u)(t,\cdot)\Vert_p &=\left\Vert \sum_{j=1}^{d} \int_{0}^{t} \int_{\R^d} \left[\frac{\partial}{\partial y_j } p(t-s,\cdot-y)\right]  q_j(s,y,\pi_n u(s, y))\ud s\ud y \right\Vert_p\\
  &\le \sum_{j=1}^{d} \left\Vert  \int_{0}^{t} \int_{\R^d} \left[\frac{\partial}{\partial y_j } p(t-s,\cdot-y)\right]  q_j(s,y,\pi_n u(s, y))\ud s\ud y \right\Vert_p\\
  &\le \sum_{j=1}^{d}   \int_{0}^{t} \left\Vert\int_{\R^d} \left[\frac{\partial}{\partial y_j } p(t-s,\cdot-y)\right]  q_j(s,y,\pi_n u(s, y)) \ud y\right\Vert_p \ud s.
\end{align*}
By Young's convolution inequality and \eqref{e:con-f-p}, we get
\begin{align*}
  \Vert (\mathcal T_2 u)(t,\cdot)\Vert_p & \le \sum_{j=1}^{d} \int_{0}^{t} \left(\int_{\R^d} \left|\frac{\partial}{\partial x_j } p(t-s,x) \right|\ud x\right)
  \left\Vert  q_j(s,\cdot,\pi_n u(s, \cdot))\right\Vert_p   \ud s\\
  &\le C \sum_{j=1}^{d} \int_{0}^{t} \left(\int_{\R^d} \left|\frac{\partial}{\partial x_j } p(t-s,x) \right|\ud x\right)
  \left\Vert a_1(\cdot)+ |\pi_n u(s,\cdot)|\right\Vert_p   \ud s.
\end{align*}
Note that \eqref{e:con-mainth2} implies \eqref{e-con-fs-partial}. Then,  we can apply Lemma \ref{le:upb-fms-deri} to see that
\begin{align*}
   \Vert (\mathcal T_2 u)(t,\cdot)\Vert_p & \le C \int_{0}^{t} (t-s)^{\beta+\gamma-1-\beta/\alpha}  \ud s \times(\left\Vert a_1(\cdot)\right\Vert_p+n)<\infty,
\end{align*}
where the above integral in finite due to the fact that \eqref{e:con-mainth2} implies $\beta+\gamma-1-\beta/\alpha>-1$.

Now, we treat $(\mathcal T_3 u)(t,x)$. By Fubini's theorem and  Proposition \ref{prop:boundpoisson}, we obtain that for all $t\in[0,T]$
\begin{align*}
   & \E\left[\left\Vert(\mathcal T_3 u)(t,x)\right\Vert^p_p \right]\\
   &=\int_{\R^d} \E\left[\left|\int_{0}^{t}\int_{\R^d}\int_E p(t-s,x-y) h(s,y,\pi_n u(s,y);\xi)M(\ud s,\ud y,\ud\xi) \right|^p \right]\ud x\\
   &\le C \E\left[\int_{\R^d}\int_{0}^{t}\int_{\R^d}\int_E \left|p(t-s,x-y) h(s,y,\pi_n u(s,y);\xi)\right|^p \ud y\ud s\mu(\ud\xi)\ud x \right].
\end{align*}
Then, using \eqref{e:con-sigma-h-p} and Lemma \ref{le:upb-fmsl}, we have
\begin{align*}
   & \int_{\R^d}\int_{0}^{t}\int_{\R^d}\int_E \left|p(t-s,x-y) h(s,y,\pi_n u(s,y);\xi)\right|^p \ud y\ud s\mu(\ud\xi)\ud x\\
   &\le C \int_{\R^d}\int_{0}^{t}\int_{\R^d} \left|p(t-s,x-y)\right|^p \left(a_2(y)+\left|\pi_n u(s,y) \right|^p\right)\ud y\ud s\ud x\\
   &\le C \int_{0}^{t}\int_{\R^d} |t-s|^{p(\beta+\gamma-1)+\frac{\beta d}{\alpha}(1-p)}\left(a_2(y)+\left|\pi_n u(s,y) \right|^p\right)\ud y\ud s\\
   &\le C \int_{0}^{t}|t-s|^{p(\beta+\gamma-1)+\frac{\beta d}{\alpha}(1-p)} \left(\Vert a\Vert_1 +n^p\right)\ud s.
\end{align*}
Then, for any $t\in[0,T]$, we have
\begin{align}\label{e:j3-infty}
   \E\left[\left\Vert(\mathcal T_3 u)(t,x)\right\Vert^p_p\right] & \le C \int_{0}^{t}|t-s|^{p(\beta+\gamma-1)+\frac{\beta d}{\alpha}(1-p)} \left(\Vert a\Vert_1+n^p\right)\ud s \notag\\
   &\le C t^{1+p(\beta+\gamma-1)+\frac{\beta d}{\alpha}(1-p)}<\infty,
\end{align}
where the last inequality holds due to the fact that \eqref{e:con-mainth2} implies
\begin{equation*}
  p(\beta+\gamma-1)+\frac{\beta d}{\alpha}(1-p)>-1.
\end{equation*}
Therefore, combining \eqref{e:de-cT}, \eqref{e:esti-j0}, \eqref{e:j1-infty} and \eqref{e:j3-infty}, we confirm that  $\mathcal T$ maps $B_{T,p}$ into itself which completes Step 1.

{\em Step 2.} We will apply the Banach fixed point argument to obtain the existence and uniqueness of the solution to \eqref{e:de-solution2}. Let $\kappa>0$ be an arbitrary fixed number. Let $B_{p,\kappa}$ be the space consisting of all $\cF_t$-adapted process $u:[0,T]\times\R^d\mapsto u(t,x)\in\R$ with norm
\begin{equation}\label{e:de=norm2}
  \Vert u\Vert_{p,\kappa}:=\sup_{t\in[0,T]} e^{-\kappa t}\E\left[\left\Vert u(t,\cdot)\right\Vert^p_{p}\right]^{1/p}<\infty.
\end{equation}
Note that the new norm $\Vert u\Vert_{p,\kappa}$ is equivalent to the previous norm $\Vert u\Vert_{T,p}$ for any fixed $\kappa>0$. Then, we can prove that $\mathcal T :B_{p,\kappa} \to B_{p,\kappa}$ is well defined by the arguments in Step 1. Next, we aim to show that the mapping $\mathcal{T}$ is a contraction for sufficiently large $\kappa$. More precisely, for any $u,v\in B_{p,\kappa}$, there exists a constant $c\in(0,1)$ such that
\begin{equation*}
  \Vert \mathcal{T}u-\mathcal{T}v\Vert_{p,\kappa}\le c \Vert u-v\Vert_{p,\kappa}.
\end{equation*}

From \eqref{e:con-f2-p} and Minkowski inequality, we get
\begin{align*}
   & \Vert \mathcal{T}_1 u-\mathcal{T}_1 v\Vert_{p,\kappa}^p\\
   &=\sup_{t\in[0,T]} e^{-\kappa pt}\E\bigg[\left\Vert \int_{0}^{t}\int_{\R^d} |p(t-s,\cdot-y)| \left[ f(s,y,\pi_n u(s,y))-f(s,y,\pi_n v(s,y)) \right]\ud s \ud y \right\Vert_p^p \bigg]\\
   &\le C\sup_{t\in[0,T]} e^{-\kappa pt}\E\bigg[\bigg\Vert \int_{0}^{t}\int_{\R^d} |p(t-s,\cdot-y)| \left[a_3(y)+C(|\pi_n u(s,y)| +|\pi_n v(s,y)| \right]\\
   &\qquad\qquad\qquad\qquad\qquad\qquad\qquad\qquad\qquad\qquad\qquad\times |\pi_n u(s,y)-\pi_n v(s,y)| \ud s \ud y \bigg\Vert_p^p\bigg]\\
   &\le C\sup_{t\in[0,T]} e^{-\kappa pt}\E\bigg[\bigg(\int_{0}^{t}\bigg\Vert \int_{\R^d} |p(t-s,\cdot-y)| \left[a_3(y)+C(|\pi_n u(s,y)| +|\pi_n v(s,y)| )\right]\\
   &\qquad\qquad\qquad\qquad\qquad\qquad\qquad\qquad\qquad\qquad\qquad\times |\pi_n u(s,y)-\pi_n v(s,y)| \ud y \bigg\Vert_p \ud s\bigg)^p\bigg] .
\end{align*}
Then, using Young's convolution inequality and the above estimate, we have
\begin{align*}
   & \Vert \mathcal{T}_1 u-\mathcal{T}_1 v\Vert_{p,\kappa}^p\\
   &\le C\sup_{t\in[0,T]} e^{-\kappa pt}\E\bigg[\bigg(\int_{0}^{t} \left(\int_{\R^d} |p(t-s,x)| \ud x\right)\\
    & \times\bigg\Vert\left[a_3(\cdot)+C(|\pi_n u(s,\cdot)| +|\pi_n v(s,\cdot)| )\right] |\pi_n u(s,\cdot)-\pi_n v(s,\cdot)|\bigg\Vert_p  \ud s\bigg)^p\bigg] .
\end{align*}
Under \eqref{e-con-fs}, we can apply Lemma \ref{le:upb-fmsl} to get
\begin{align*}
  & \Vert \mathcal{T}_1 u-\mathcal{T}_1 v\Vert_{p,\kappa}^p\\
   &\le C\sup_{t\in[0,T]} e^{-\kappa pt}\E\bigg[\bigg(\int_{0}^{t} (t-s)^{\beta+\gamma-1}\\
    & \times\bigg\Vert\left[a_3(\cdot)+C(|\pi_n u(s,\cdot)|+|\pi_n v(s,\cdot)|)\right] |\pi_n u(s,\cdot)-\pi_n v(s,\cdot)|\bigg\Vert_p  \ud s\bigg)^p\bigg]
\end{align*}
Applying \eqref{e:pin-contract} and Jensen's inequality, we have
\begin{align*}
  & \Vert \mathcal{T}_1 u-\mathcal{T}_1 v\Vert_{p,\kappa}^p\\
   &\le C\sup_{t\in[0,T]} e^{-\kappa pt}\E\bigg[\bigg(\int_{0}^{t} (t-s)^{\beta+\gamma-1}\left\Vert u(s,\cdot)- v(s,\cdot)\right\Vert_p  \ud s\bigg)^p\bigg] \times\left(\Vert a_3(\cdot)\Vert_p^p+2n^p\right)\\
   &\le C\sup_{t\in[0,T]} e^{-\kappa pt}\int_{0}^{t} (t-s)^{p(\beta+\gamma-1)} \E\left[\left\Vert u(s,\cdot)- v(s,\cdot)\right\Vert_p^p\right]  \ud s \\
   &= C \sup_{t\in[0,T]} \int_{0}^{t}e^{-\kappa p(t-s)} (t-s)^{p(\beta+\gamma-1)} e^{-\kappa ps}\E\left[\left\Vert u(s,\cdot)- v(s,\cdot)\right\Vert_p^p\right]  \ud s\\
   &\le C\Vert u-v\Vert_{p,\kappa}^p  \int_{0}^{T}e^{-\kappa ps} s^{p(\beta+\gamma-1)} \ud s.
\end{align*}

Next, using Minkowski inequality, we have
\begin{align*}
   & \Vert \mathcal{T}_2 u-\mathcal{T}_2 v\Vert_{p,\kappa}^p\\
   &= \sup_{t\in[0,T]} e^{-\kappa pt} \E\bigg[\bigg\Vert \sum_{j=1}^{d} \int_{0}^{t} \int_{\R^d} \left(\frac{\partial}{\partial y_j } p(t-s,\cdot-y)\right) \\
   &\qquad \times \left[ q_j(s,y,\pi_n u(s, y))-q_j(s,y,\pi_n v(s, y)) \right] \ud s\ud y \bigg\Vert_p^p\bigg]\\
   &\le  \sup_{t\in[0,T]} e^{-\kappa pt} \E\bigg[ \bigg(\sum_{j=1}^{d}   \int_{0}^{t} \bigg\Vert\int_{\R^d} \left[\frac{\partial}{\partial y_j } p(t-s,\cdot-y)\right] \\
    &\qquad \times \left[ q_j(s,y,\pi_n u(s, y))-q_j(s,y,\pi_n v(s, y)) \right] \ud y\bigg\Vert_p \ud s\bigg)^p\bigg].
\end{align*}
Applying Young's convolution inequality and \eqref{e:con-f2-p} to obtain that
\begin{align*}
  & \Vert \mathcal{T}_2 u-\mathcal{T}_2 v\Vert_{p,\kappa}^p\\
   &\le \sup_{t\in[0,T]} e^{-\kappa pt} \E\bigg[ \bigg( \sum_{j=1}^{d} \int_{0}^{t} \left(\int_{\R^d} \left|\frac{\partial}{\partial x_j } p(t-s,x) \right|\ud x\right)\\
  &\qquad\times \left\Vert  q_j(s,y,\pi_n u(s, y))-q_j(s,y,\pi_n v(s, y))\right\Vert_p   \ud s\bigg)^p\bigg] \\
  &\le \sup_{t\in[0,T]} e^{-\kappa pt} \E\bigg[ \bigg( \sum_{j=1}^{d} \int_{0}^{t} \left(\int_{\R^d} \left|\frac{\partial}{\partial x_j } p(t-s,x) \right|\ud x\right)\\
  &\qquad\times \bigg\Vert\left[a_3(\cdot)+C(|\pi_n u(s,\cdot)| +|\pi_n v(s,\cdot)| )\right] |\pi_n u(s,\cdot)-\pi_n v(s,\cdot)|\bigg\Vert_p   \ud s\bigg)^p\bigg] .
\end{align*}
By Lemma \ref{le:upb-fms-deri}, we have
\begin{align*}
  & \Vert \mathcal{T}_2 u-\mathcal{T}_2 v\Vert_{p,\kappa}^p\\
   &\le C\sup_{t\in[0,T]} e^{-\kappa pt} \E\bigg[ \bigg( \int_{0}^{t} (t-s)^{\beta+\gamma-1-\beta/\alpha}\\
  &\qquad\times \bigg\Vert\left[a_3(\cdot)+C(|\pi_n u(s,\cdot)| +|\pi_n v(s,\cdot)| )\right] |\pi_n u(s,\cdot)-\pi_n v(s,\cdot)|\bigg\Vert_p   \ud s\bigg)^p\bigg] .
\end{align*}
Applying \eqref{e:pin-contract} and H\"older's inequality with $q=\frac{p}{p-1}>1$, we have
\begin{align*}
   & \Vert \mathcal{T}_2 u-\mathcal{T}_2 v\Vert_{p,\kappa}^p\\
   &\le C\sup_{t\in[0,T]} e^{-\kappa pt} \E\bigg[ \bigg( \int_{0}^{t} (t-s)^{\beta+\gamma-1-\beta/\alpha} \Vert  u(s,\cdot)- v(s,\cdot) \Vert_p   \ud s\bigg)^p\bigg] \times  \left(\Vert a_3(\cdot)\Vert_p^p+2n^p \right)\\
   &\le C\sup_{t\in[0,T]} e^{-\kappa pt} \E\bigg[ \bigg(\int_{0}^{t} (t-s)^{(\beta+\gamma-1-\beta/\alpha)/p+(\beta+\gamma-1-\beta/\alpha)/q } \Vert  u(s,\cdot)- v(s,\cdot) \Vert_p   \ud s\bigg)^p\bigg] \\
   &\le C\sup_{t\in[0,T]} e^{-\kappa pt} \E\bigg[ \int_{0}^{t} (t-s)^{\beta+\gamma-1-\beta/\alpha} \Vert  u(s,\cdot)- v(s,\cdot) \Vert_p^p   \ud s\bigg]\times \bigg(\int_{0}^{t} t^{\beta+\gamma-1-\beta/\alpha }\ud s\bigg)^{q/p}.
\end{align*}
Since \eqref{e:con-mainth2} implies $\beta+\gamma-1-\beta/\alpha>-1$, we have
\begin{align*}
   &  \Vert \mathcal{T}_2 u-\mathcal{T}_2 v\Vert_{p,\kappa}^p\\
   &\le C \sup_{t\in[0,T]} \int_{0}^{t}e^{-\kappa p(t-s)} (t-s)^{\beta+\gamma-1-\beta/\alpha} e^{-\kappa ps}\E\left[\left\Vert u(s,\cdot)- v(s,\cdot)\right\Vert_p^p\right]  \ud s\\
   &\le C\Vert u-v\Vert_{p,\kappa}^p  \int_{0}^{T}e^{-\kappa ps} s^{\beta+\gamma-1-\beta/\alpha} \ud s.
\end{align*}

Then, by Fubini's theorem and  Proposition \ref{prop:boundpoisson}, we have
\begin{align*}
   & \Vert \mathcal{T}_3 u-\mathcal{T}_3 v\Vert_{p,\kappa}^p\\
   &=\sup_{t\in[0,T]} e^{-\kappa pt}\int_{\R^d} \ud x\\
   &\quad\times \E\left[\left|\int_{0}^{t}\int_{\R^d}\int_E p(t-s,x-y) [h(s,y,\pi_n u(s,y);\xi)-h(s,y,\pi_n v(s,y);\xi)] M(\ud s,\ud y,\ud\xi) \right|^p \right]\\
   &\le C \sup_{t\in[0,T]} e^{-\kappa pt}\int_{\R^d} \ud x\\
   &\quad\times\E\left[\int_{0}^{t}\int_{\R^d}\int_E \left|p(t-s,x-y) [h(s,y,\pi_n u(s,y);\xi)-h(s,y,\pi_n v(s,y);\xi)]\right|^p \ud y\ud s\mu(\ud\xi) \right].
\end{align*}
Then, using \eqref{e:con-h2-p}, we have
\begin{align*}
   & \Vert \mathcal{T}_3 u-\mathcal{T}_3 v\Vert_{p,\kappa}^p\\
   &\le C \sup_{t\in[0,T]} e^{-\kappa pt}\int_{\R^d}\E\left[\int_{0}^{t}\int_{\R^d}\big|p(t-s,x-y) [\pi_n u(s,y)-\pi_n v(s,y)]\big|^p \ud y\ud s\right] \ud x.
\end{align*}
 Applying the Young's inequality for integral operators corresponding to $x,y$ to obtain that
 \begin{align*}
   & \Vert \mathcal{T}_3 u-\mathcal{T}_3 v\Vert_{p,\kappa}^p\\
   &\le C \sup_{t\in[0,T]} e^{-\kappa pt}\E\left[\int_{0}^{t}\ud s \times\int_{\R^d}\left|p(t-s,x)\right|^p \ud x \times\int_{\R^d}\left| \pi_n u(s,y)-\pi_n v(s,y)\right|^p \ud y \right].
 \end{align*}
 By Lemma \ref{le:upb-fmsl} and \eqref{e:pin-contract}, we have
 \begin{align*}
   & \Vert \mathcal{T}_3 u-\mathcal{T}_3 v\Vert_{p,\kappa}^p\\
   &\le C \sup_{t\in[0,T]} e^{-\kappa pt} \int_{0}^{t}|t-s|^{p(\beta+\gamma-1)+\frac{\beta d}{\alpha}(1-p)} \E\left[ \left\Vert u(s,\cdot)-v(s,\cdot)\right\Vert_p^p \right]\ud s\\
   &= C \sup_{t\in[0,T]}  \int_{0}^{t}e^{-\kappa p(t-s)}|t-s|^{p(\beta+\gamma-1)+\frac{\beta d}{\alpha}(1-p)} e^{-\kappa ps} \E\left[\left\Vert u(s,\cdot)-v(s,\cdot)\right\Vert_p^p \right]\ud s\\
   &\le C \Vert u-v\Vert_{p,\kappa}^p \int_{0}^{T}e^{-\kappa ps}s^{p(\beta+\gamma-1)+\frac{\beta d}{\alpha}(1-p)}\ud s.
 \end{align*}
 Since condition \eqref{e:con-mainth2} holds, we can choose a sufficiently large $\kappa>0$ such that
 \begin{equation*}
   C\int_{0}^{T}e^{-\kappa ps} \left(s^{p(\beta+\gamma-1)}+s^{\beta+\gamma-1-\beta/\alpha} +s^{p(\beta+\gamma-1)+\frac{\beta d}{\alpha}(1-p)}\right) \ud s<1.
 \end{equation*}
We have confirmed that the mapping $\mathcal{T}:B_{p,\kappa} \to B_{p,\kappa}$ is a contraction, and hence, $\mathcal{T}$ has a unique fixed point $u(t,x)$ in $B_{p,\kappa}$ which is the solution to the truncated stochastic integral equation \eqref{e:de-truncatedSIE} for any given $n\in\N$.

{\em Step 3.} Here we will construct a local solution to \eqref{e:de-solution2}. We denote by $u_n$ the unique solution to \eqref{e:de-truncatedSIE} and define the $\cF_t$ stopping time
\begin{equation*}
  \tau_n:=\inf\left\{ t\in[0,T]:\left\Vert u_n(t,\cdot)\right\Vert_p \ge n \right\}.
\end{equation*}
It is clear that $\{\tau_n\}_{n\in\N}$ is an increasing sequence and hence the limit $\tau_\infty:=\lim_{n\to\infty}\tau_n$ exists. By the construction of the unique solution $u_n$, we have that for any $m>n$
\begin{equation*}
  u_m(t,x)=u_n(t,x), \quad t<\tau_n,\:x\in\R^d  \text{ a.s.}
\end{equation*}
We define
\begin{equation*}
  u(t,x):=u_n(t,x), \quad (t,x)\in [0,\tau_n)\times\R^d,
\end{equation*}
and then the local solution to \eqref{e:de-solution2} is
\begin{equation*}
  \{u(t,x):(t,x)\in[0,\tau_\infty)\times\R^d\}.
\end{equation*}

 Then we prove the uniqueness to finish this step. Suppose that there are two local solution $u$ and $v$ to Eq. \eqref{e:de-solution2}. Thus, $u$ and $v$ must satisfy Eq. \eqref{e:de-truncatedSIE}for any fixed $n\in\N$. On the other hand, by the uniqueness of the solution to Eq. \eqref{e:de-truncatedSIE}, we have $u(t,x)=v(t,x)$,
 for all $(t,x)\in[0,\tau_n)\times\R^d$. Now, let $n\to\infty$, we get $u(t,x)=v(t,x)$,
  for all $(t,x)\in[0,\tau_\infty)\times\R^d$. Hence, we confirm the uniqueness.

{\em Step 4.} Here, we will show that there exists a predictable modification for the solution $u(t)$. By \cite[Proposition 3.21]{Peszat2007Stochastic}, we know that any measurable stochastically continuous $\cF_t$-adapted process has a predictable modification. Thus, it is sufficient to show that
\begin{align*}
  \lim_{t'\downarrow t}\E\bigg[\bigg| & \int_{0}^{t'}\int_{\R^d}\int_E p(t'-s,x-y) h\left(s,y,u(s,y);\xi\right) M(\ud s,\ud y,\ud\xi) \\
       &-\int_{0}^{t}\int_{\R^d}\int_E p(t-s,x-y)h\left(s,y,u(s,y);\xi\right) M(\ud s,\ud y,\ud\xi)\bigg|^p\bigg]=0.
\end{align*}
By dominated convergence theorem, it is sufficient to prove that
\begin{equation}\label{e:sthca-contin}
  \begin{split}
     \int_{\R^d}\E\bigg[\bigg| & \int_{0}^{t'}\int_{\R^d}\int_E p(t'-s,x-y) h\left(s,y,u(s,y);\xi\right) M(\ud s,\ud y,\ud\xi) \\
       &-\int_{0}^{t}\int_{\R^d}\int_E p(t-s,x-y)h\left(s,y,u(s,y);\xi\right) M(\ud s,\ud y,\ud\xi)\bigg|^p\bigg]\ud x
  \end{split}
\end{equation}
 converges to 0 when $t'\downarrow t$. Using Proposition \ref{prop:boundpoisson}, \eqref{e:con-sigma-h-p} and Fubini's theorem, we have \eqref{e:sthca-contin} is bounded above by
\begin{align*}
   I_1+I_2:=&C\E\left[\int_{\R^d}\int_{0}^{t}\int_{\R^d}\left|p(t'-s,x-y)-p(t-s,x-y)\right|^p\left(a_2(y)+|u(s,y)|^p\right)\ud y\ud s\ud x\right]\\
   &+C\E\left[\int_{\R^d}\int_{t}^{t'}\int_{\R^d}\left|p(t'-s,x-y)\right|^p\left(a_2(y)+|u(s,y)|^p\right)\ud y\ud s\ud x\right].
\end{align*}
For the second term $I_2$, using Lemma \ref{le:upb-fmsl}, we have
\begin{align*}
  I_2 &\le C\int_{t}^{t'}\int_{\R^d} |t-s|^{p(\beta+\gamma-1)+\frac{\beta d}{\alpha}(1-p)}\left(a_2(y)+\left| u(s,y) \right|^p\right)\ud y\ud s\\
   &\le C \int_{t}^{t'}|t-s|^{p(\beta+\gamma-1)+\frac{\beta d}{\alpha}(1-p)}\ud s\times
   \left(\Vert a_2(\cdot)\Vert_1 +\sup_{r\in[0,T]}\E\left[\left\Vert u(r,\cdot)\right\Vert^p_p \right]\right),
\end{align*}
which converges to 0 when $t'\downarrow t$ since the condition \eqref{e:con-mainth2} holds. For the first term $I_1$, the integrand is bounded above by
\begin{equation*}
  2^{p-1}\left(\left|p(t'-s,x-y)\right|^p+\left|p(t-s,x-y)\right|^p\right) \left(a_2(y)+|u(s,y)|^p\right),
\end{equation*}
 and then we also have
\begin{align*}
  I_1 &\le C\E\left[\int_{\R^d}\int_{0}^{t}\int_{\R^d} 2^{p-1}\left(\left|p(t'-s,x-y)\right|^p+\left|p(t-s,x-y)\right|^p\right) \left(a_2(y)+|u(s,y)|^p\right)\ud y\ud s\ud x\right]\\
  &\le C2^{p-1}\int_{0}^{t}\left(|t'-s|^{p(\beta+\gamma-1)+\frac{\beta d}{\alpha}(1-p)}+|t-s|^{p(\beta+\gamma-1)+\frac{\beta d}{\alpha}(1-p)} \right)\ud s\\
  &\qquad\qquad\qquad\times \left(\Vert a_2(\cdot)\Vert_1 +\sup_{r\in[0,T]}\E\left[\left\Vert u(r,\cdot)\right\Vert^p_p \right]\right)<\infty.
\end{align*}
Thus, by  dominated convergence theorem, we have
\begin{align}\label{finalEq}
  &\lim_{t'\downarrow t}  \E\left[\int_{\R^d}\int_{0}^{t}\int_{\R^d}\left|p(t'-s,x-y)-p(t-s,x-y)\right|^p\left(f(y)+|u(s,y)|^p\right)\ud y\ud s\ud x\right]\nonumber\\
  &=\E\left[\int_{\R^d}\int_{0}^{t}\int_{\R^d}\lim_{t'\downarrow t}\left|p(t'-s,x-y)-p(t-s,x-y)\right|^p\left(f(y)+|u(s,y)|^p\right)\ud y\ud s\ud x\right]=0.
\end{align}
Therefore, we have proved  the solution has a predictable modification.

For the case of $p=1$, one can derive directly through the above steps with the corresponding key inequality \eqref{e:j3-infty} and other relevant inequalities (without utilising Young’s convolution inequality and H\"older's inequality),  as well as the final limit equality \eqref{finalEq} and so on.
Hence, Theorem \ref{th:purejump} does valid for $p=1$. The proof is then complete.
\end{proof}

Next, we prove theorem \ref{th:1}.

\begin{proof}[Proof of Theorem \ref{th:1}]
  For each $n\in\N$, let $\lambda_n$ be the mapping from $L^2(\R^d)$ to $L^2(\R^d)$, i.e.,
\begin{equation*}
  \lambda_n u(t,x)=
  \begin{cases}
    u(t,x), & \mbox{if } \Vert u\Vert_2 \le n \\
    \frac{n}{\Vert u(t,\cdot)\Vert_2} u(t,x), & \mbox{if } \Vert u\Vert_2 \ge n.
  \end{cases}
\end{equation*}
Clearly, we have $\lambda_n u\le n$ and the mapping $\lambda_n:L^2(\R^d)\to L^2(\R^d)$ is Lipschitz continuous, i.e.,
\begin{equation}\label{e:lambda-contract}
  \left\Vert\lambda_n u-\lambda_n v\right\Vert_2 \le \left\Vert u-v\right\Vert_2,\quad \forall u,v\in L^2(\R^d),
\end{equation}
which implies that $\lambda_n$ is also a contraction.

For any fixed $n$, the truncated stochastic integral equation associated with $\lambda_n$ is
\begin{equation}\label{e:truncated-white}
\begin{split}
   u(t,x)= & J_0(t,x)+ \int_{0}^{t}\int_{\R^d} p(t-s,x-y)f(s,y,\lambda_n u(s,y)) \ud s \ud y\\
   &+\sum_{j=1}^{d} \int_{0}^{t} \int_{\R^d} \left[\frac{\partial}{\partial y_j } p(t-s,x-y)\right]  q_j(s,y,\pi_n u(s, y))\ud s\ud y\\
        &+ \int_{0}^{t}\int_{\R^d} p(t-s,x-y)\sigma(s,y,\lambda_n u(s,y)) W(\ud s,\ud y)\\
        &+ \int_{0}^{t}\int_{\R^d}\int_{E} p(t-s,x-y) h(s,y,\lambda_n u(s,y);\xi)M(\ud s,\ud y,\ud\xi).
\end{split}
\end{equation}

Recall that the Banach space $B_{T,2}$ with norm $\Vert \cdot\Vert_{T,2}$ is defined in \eqref{e:de-norm1}. Now we define a mapping $\mathcal J$ on $B_{T,2}$:
\begin{equation*}
  (\mathcal J u)(t,x) := J_0(t,x) + \sum_{i=1}^{4}(\mathcal J_i u)(t,x),
\end{equation*}
with
\begin{align*}
  (\mathcal J_1 u)(t,x) & =  \int_{0}^{t}\int_{\R^d} p(t-s,x-y)f(s,y,\lambda_n u(s,y)) \ud s \ud y, \\
  (\mathcal J_2 u)(t,x) & = \sum_{j=1}^{d} \int_{0}^{t} \int_{\R^d} \left[\frac{\partial}{\partial y_j } p(t-s,x-y)\right]  q_j(s,y,\pi_n u(s, y))\ud s\ud y,\\
  (\mathcal J_3 u)(t,x) & = \int_{0}^{t}\int_{\R^d}p(t-s,x-y)\sigma(s,y,\lambda_n u(s,y))W(\ud s,\ud y), \\
  (\mathcal J_4 u)(t,x) & = \int_{0}^{t}\int_{\R^d}\int_{E} p(t-s,x-y) h(s,y,\lambda_n u(s,y);\xi)M(\ud s,\ud y,\ud\xi).
\end{align*}
We first need to confirm that the mapping $\mathcal J$ maps $B_{T,2}$ into itself. By a similar argument in step 1 in the proof of Theorem \ref{th:purejump}, one can deal with $J_0$, $\mathcal J_1$, $\mathcal J_2$ and $\mathcal J_4$. And then, it is sufficient to show that $\mathcal J_3:B_{T,2}\mapsto B_{T,2}$. By It\^o isometry property and \eqref{e:con-sigma-h}, we have
\begin{align*}
   \E\left[\Vert (\mathcal J_3 u)(t,\cdot)\Vert_2^2\right]
   &=\E\left[\int_{0}^{t}\int_{\R^d}\int_{\R^d} \left|p(t-s,x-y)\sigma(s,y,\lambda_n u(s,y))\right|^2 \ud s \ud y \ud x \right]\\
   &=\E\left[\int_{0}^{t}\int_{\R^d}\int_{\R^d} \left|p(t-s,x-y)\right|^2 \left(g_2(y)+C|\lambda_n u(s,y)|^2\right) \ud s \ud y \ud x \right].
\end{align*}
Noting that condition \eqref{e:con-th1} implies $2\alpha>d$, we can apply Lemma \ref{le:upb-fmsl} to see that
\begin{align*}
  \E\left[\Vert (\mathcal J_3 u)(t,\cdot)\Vert_2^2\right]
   &\le\E\left[\int_{0}^{t} (t-s)^{2\beta+2\gamma-2-\frac{\beta d}{\alpha}} \int_{\R^d}  \left(g_2(y)+C|\lambda_n u(s,y)|^2\right)\ud y \ud s   \right]\\
   &\le \int_{0}^{t} (t-s)^{2\beta+2\gamma-2-\frac{\beta d}{\alpha}} \ud s \times \left(\Vert g_2(\cdot)\Vert_1+Cn^2\right)<\infty,
\end{align*}
where the integral with respect to $s$ is finite due to the fact that \eqref{e:con-th1} implies
\begin{equation*}
  2\beta+2\gamma-2-\frac{\beta d}{\alpha}>-1.
\end{equation*}
This prove that $\mathcal J$ maps $B_{T,2}$ into itself.

Recall that the Banach space $(B_{2,\kappa}, \Vert \cdot\Vert_{2,\kappa})$ is defined in \eqref{e:de=norm2}. The new norm $\Vert \cdot\Vert_{2,\kappa}$ is equivalent to the previous norm $\Vert \cdot\Vert_{T,2}$ for any fixed $\kappa>0$. Next, we aim to show that the mapping $\mathcal{J}:B_{2,\kappa}\mapsto B_{2,\kappa}$ is a contraction, i.e., there exists a constant $c\in(0,1)$ such that
\begin{equation*}
  \Vert \mathcal{J}u-\mathcal{J}v\Vert_{2,\kappa}\le c \Vert u-v\Vert_{2,\kappa},
\end{equation*}
for any $u,v\in B_{2,\kappa}$.

By It\^o isometry property and \eqref{e:con-sigma-h}, we have
\begin{align*}
   &\Vert \mathcal{J}_3 u-\mathcal{J}_3 v\Vert_{2,\kappa}=\sup_{t\in[0,T]} e^{-\kappa t} \E\left[\Vert \mathcal J_3 u - \mathcal J_3 v\Vert_2^2\right]^{1/2}\\
   &=\sup_{t\in[0,T]} e^{-\kappa t} \E\left[\int_{0}^{t}\int_{\R^d}\int_{\R^d} \left|p(t-s,x-y)[\sigma(s,y,\lambda_n u(s,y))-\sigma(s,y,\lambda_n v(s,y))]\right|^2 \ud s \ud y \ud x \right]^{1/2}\\
   &\le C\sup_{t\in[0,T]} e^{-\kappa t} \E\left[\int_{0}^{t}\int_{\R^d}\int_{\R^d} \left|p(t-s,x-y)\right|^2 \times |\lambda_n u(s,y)-\lambda_n v(s,y)|^2 \ud s \ud y \ud x \right]^{1/2}.
\end{align*}
Under condition \eqref{e:con-th1}, we can apply Lemma \ref{le:upb-fmsl} to see that
\begin{align*}
   &\Vert \mathcal{J}_3 u-\mathcal{J}_3 v\Vert_{2,\kappa}\\
   &\le C\sup_{t\in[0,T]} e^{-\kappa t} \E\left[\int_{0}^{t} (t-s)^{2\beta+2\gamma-2-\frac{\beta d}{\alpha}} \int_{\R^d}  |\lambda_n u(s,y)-\lambda_n v(s,y)|^2 \ud y  \ud s  \right]^{1/2}.
\end{align*}
By \eqref{e:lambda-contract}, we have
\begin{align*}
  &\Vert \mathcal{J}_3 u-\mathcal{J}_3 v\Vert_{2,\kappa}\\
   &\le C\sup_{t\in[0,T]} e^{-\kappa t} \E\left[\int_{0}^{t} (t-s)^{2\beta+2\gamma-2-\frac{\beta d}{\alpha}} \Vert u(s,\cdot)-v(s,\cdot) \Vert_2^2  \ud s  \right]^{1/2}\\
   &= C\sup_{t\in[0,T]} \E\left[\int_{0}^{t} e^{-2\kappa (t-s)} (t-s)^{2\beta+2\gamma-2-\frac{\beta d}{\alpha}} \times  e^{-2\kappa s} \Vert u(s,\cdot)-v(s,\cdot) \Vert_2^2 \ud s  \right]^{1/2}\\
   &\le C \sup_{t\in[0,T]} \left(\int_{0}^{t} e^{-2\kappa (t-s)} (t-s)^{2\beta+2\gamma-2-\frac{\beta d}{\alpha}}  \ud s  \right)^{1/2}\times \Vert u- v\Vert_{2,\kappa},
\end{align*}
where
\begin{equation*}
  C \sup_{t\in[0,T]} \left(\int_{0}^{t} e^{-2\kappa (t-s)} (t-s)^{2\beta+2\gamma-2-\frac{\beta d}{\alpha}}  \ud s  \right)^{1/2}<1,
\end{equation*}
for sufficiently large $\kappa$. Using the similar approach in step 2 in the proof of Theorem \ref{th:purejump}, we can also prove that for some $c<1$,
\begin{equation*}
  \Vert \mathcal{J}_1 u-\mathcal{J}_1 v\Vert_{2,\kappa} \le c \Vert u- v\Vert_{2,\kappa},
\end{equation*}
\begin{equation*}
  \Vert \mathcal{J}_2 u-\mathcal{J}_2 v\Vert_{2,\kappa}\le c \Vert u- v\Vert_{2,\kappa},
\end{equation*}
and
\begin{equation*}
  \Vert \mathcal{J}_4 u-\mathcal{J}_4 v\Vert_{2,\kappa}\le c \Vert u- v\Vert_{2,\kappa}.
\end{equation*}
Therefore, the mapping  $\mathcal{J}:B_{2,\kappa}\mapsto B_{2,\kappa}$ is a contraction. Hence, there must be a unique fixed point in $B_{2,\kappa}$ for $\mathcal{J}$, and this point is the unique solution for Eq. \eqref{e:truncated-white} for each $n\in\N$. For this $u_n$, let us define the stopping time
\begin{equation*}
  \tau_n(\omega):=\inf\left\{ t\in[0,T]:\int_{\R^d}u_n^2(t,x)\ud x\ge n^2 \right\}.
\end{equation*}
Note that $\{\tau_n\}_{n\in\N}$ is an increasing sequence. Clearly by the contraction property of $\mathcal{J}$, we have for all $m\ge n$,
\begin{equation*}
  u_m(t,x)=u_n(t,x), \quad \forall t\in[0,\tau_n),\: x\in\R^d,
\end{equation*}
almost surely. We define
\begin{equation*}
  u(t,x):=u_n(t,x), \quad (t,x)\in [0,\tau_n)\times\R^d,
\end{equation*}
and
\begin{equation*}
  \tau_{\infty}:=\lim_{n\to\infty}\tau_n.
\end{equation*}
Then, the local solution to Eq. \eqref{e:de-solution} is
\begin{equation*}
  \{u(t,x):(t,x)\in[0,\tau_\infty)\times\R^d\}.
\end{equation*}
Moreover, the uniqueness of the local solution $u$ to Eq. \eqref{e:de-solution} is directly follows from the uniqueness of the solution $u_n$ to \eqref{e:truncated-white}.

Using the similar approach, one can also prove that there exists a predictable modification for the solution $u(t)$. This completes the proof.

\end{proof}

\noindent{\bf Acknowledgements.}

This work is supported by the Guangdong Provincial Key Laboratory of Interdisciplinary Research
and Application for Data Science under project No. UIC 2022B1212010006, by the UIC Start-up Research
Fund (No. UICR0700072-24), and by Guangdong and Hong Kong Universities “1+1+1” Joint Research Collaboration
Scheme (project No. : UICR0800012-24 and project No.: UICR0800012-24A).


\end{document}